\documentclass[a4paper,english,leqno]{amsart}

\title[Semilinear Hyperbolic SPDEs with polynomially bounded coefficients]{Solution theory to\\
Semilinear Hyperbolic Stochastic Partial Differential Equations\\
with polynomially bounded coefficients
}

\author{Alessia Ascanelli}
\address{Dipartimento di Matematica ed Informatica, Universit\`a di Ferrara, Via Machiavelli n.~30, 44121 Ferrara, Italy}
\email{alessia.ascanelli@unife.it}

\author{Sandro Coriasco}
\address{Dipartimento di Matematica ``G. Peano'', Universit\`a degli Studi di Torino, via Carlo Alberto n.~10, 10123 Torino, Italy}
\email{sandro.coriasco@unito.it}

\author{Andr{\'e} S{\"u}\ss}
\address{C/O Dipartimento di Matematica ed Informatica, Universit\`a di Ferrara, Via Machiavelli n.~30, 44121 Ferrara, Italy}
\email{suess.andre@web.de}

\date{}

\usepackage[utf8]{inputenc}
\usepackage{mathrsfs}
\usepackage{babel}
\usepackage{amsmath,amsthm,amssymb,xy,amsxtra,latexsym,verbatim}
\usepackage{hyperref}
\usepackage{enumitem}
\usepackage{pdfsync}
\usepackage{colortbl}
\usepackage{url}

\usepackage{amsxtra}
\usepackage{pxfonts}

\newcommand*{\e}{\mathrm{e}}
\newcommand*{\ii}{\mathrm{i}}

\newcommand*{\scrF}{\ensuremath{\mathscr{F}}}	
\newcommand*{\scrL}{\ensuremath{\mathscr{L}}}	

\newcommand*{\caD}{\ensuremath{\mathcal{D}}}	  
\newcommand*{\caE}{\ensuremath{\mathcal{E}}}		
\newcommand*{\caF}{\ensuremath{\mathcal{F}}}		
\newcommand*{\caH}{\ensuremath{\mathcal{H}}}
\newcommand*{\caS}{\ensuremath{\mathcal{S}}}		
\newcommand*{\caL}{\ensuremath{\mathcal{L}}}		
\newcommand*{\caO}{\ensuremath{\mathcal{O}}}		
\newcommand*{\caM}{\ensuremath{\mathcal{M}}}		

\def\ds{\displaystyle}

\newcommand*{\N}{\mathbb{N}}										
\newcommand*{\R}{\mathbb{R}}										
\newcommand*{\Rd}{{\mathbb{R}^d}}							
\newcommand*{\C}{\mathbb{C}}

\newcommand*{\E}{\mathbb{E}}										
\renewcommand*{\P}{\mathbb{P}}									
	
\newcommand{\x}{\langle x\rangle}
\newcommand{\csi}{\langle \xi \rangle}

\newcommand{\pdd}{\langle D \rangle}
\newcommand{\jap}{\langle \cdot\rangle}
\def\<{{\langle}}
\def\>{{\rangle}}
\def\perm{{\varpi}}

\newcommand{\Lip}{\mathrm{Lip}}
\newcommand{\Liploc}{\mathrm{Lip}_{\mathrm{loc}}}
\newcommand{\rf}{\mathrm{rf}}
\newcommand{\fv}{\mathrm{fv}}

\numberwithin{equation}{section}
\allowdisplaybreaks

\theoremstyle{plain}
\newtheorem{lemma}{Lemma}[section]
\newtheorem{theorem}[lemma]{Theorem}
\newtheorem{proposition}[lemma]{Proposition}
\newtheorem{corollary}[lemma]{Corollary}

\theoremstyle{definition}
\newtheorem{definition}[lemma]{Definition}
\newtheorem{remark}[lemma]{Remark}
\newtheorem{example}[lemma]{Example}

\newcommand{\afrac}[2]{\genfrac{}{}{0pt}{1}{#1}{#2}}

\newcommand{\beqsn}{\arraycolsep1.5pt\begin{eqnarray*}}
\newcommand{\eeqsn}{\end{eqnarray*}\arraycolsep5pt}
\newcommand{\beqs}{\arraycolsep1.5pt\begin{eqnarray}}
\newcommand{\eeqs}{\end{eqnarray}\arraycolsep5pt}

\usepackage{xcolor}
\definecolor{red}{rgb}{1,0,0}

\def\Op{ {\operatorname{Op}} }
\newcommand{\Ph}{\mathfrak P}
\newcommand{\Phr}{\mathfrak P_\delta}
\def\fy{\varphi}

\newcommand{\vvvert}{\|}

\begin{document}

\begin{abstract}
	We study mild solutions of a class of stochastic partial differential equations, involving operators
	with polynomially bounded coefficients. We consider semilinear equations under suitable hyperbolicity hypotheses on the linear part. We provide conditions on the initial data 
	and on the stochastic terms,
	namely, on the associated spectral measure, so that mild solutions 
	exist and are unique
	in suitably chosen functional classes. More precisely, function-valued solutions are
	obtained, as well as a regularity result.
\end{abstract}



\maketitle

%
\section{Introduction}\label{sec:intro}
The stochastic partial differential equations (SPDEs in the sequel) that we consider in the present paper are of the general form
\begin{equation}\label{eq:SPDE}
  L(t,x, \partial_t,\partial_x)u(t,x) = \gamma(t,x, u(t,x)) + \sigma(t,x, u(t,x))\dot{\Xi}(t,x),
\end{equation}
where $L$ is a linear partial differential operator that contains derivatives with respect to time ($t\in\R$) and space 
($x\in\R^d$, $d\geq 1$) variables, 
$\gamma,\ \sigma$ are real-valued functions, subject to certain regularity conditions, $\Xi$ is a random noise term that will be described in detail in Section \ref{sec:stochastics}, and $u$ is an unknown stochastic process  called \emph{solution} of the SPDE. That is, the equations \eqref{eq:SPDE} are semilinear, 
in the sense that the only possible non-linearities are on the right-hand side, and are not present in the operator $L$. In Subsection \ref{subsec:hyp}
below we will describe in more detail the conditions we impose on the operator $L$, the most important one being (a notion of) hyperbolicity.

Since the sample paths of the solution $u$ are in general not in the domain of the operator $L$, in view of the singularity of the random noise, we rewrite \eqref{eq:SPDE} in its corresponding integral (i.e., \textit{weak}) form and look for \emph{mild solutions of \eqref{eq:SPDE}}, that is, stochastic processes $u(t,x)$ satisfying
\begin{equation}\label{eq:mildsolutionSPDE}
\begin{aligned}
  u(t,x) = v_0(t,x)	&+\int_0^t\int_\Rd \Lambda(t,s,x,y)\gamma(s,y, u(s,y))dyds 	
  \\
				&+\int_0^t\int_\Rd \Lambda(t,s,x,y)\sigma(s,y, u(s,y))\dot\Xi(s,y)dyds,
\end{aligned}
\end{equation}
where:
\begin{itemize}
\item[-] $v_0$ is a deterministic term, taking into account the initial conditions; 
\item[-] $\Lambda$ is a suitable kernel, associated with the fundamental solution of the linear partial differential equation 
(linear PDE in the sequel) $Lu=0$;
\item[-] the first integral in \eqref{eq:mildsolutionSPDE} is of deterministic type, while the second is a stochastic integral. 
\end{itemize}The kind of solution $u$ we can construct for equation \eqref{eq:SPDE} depends on the approach we employ to make sense of the stochastic integral appearing in \eqref{eq:mildsolutionSPDE}. Note that both integrals in \eqref{eq:mildsolutionSPDE} contain a slight abuse of notation, 
since $\Lambda(t,s,x,y)$ is, in general, a distribution $(x,y)$. Given the commonly wide usage of such so-called 
\textit{distributional integrals}, we will also often adopt here this notation in the representation of our class of mild solutions to \eqref{eq:SPDE}. 

Classically, there are two main ways to give a meaning to stochastic integrals, like the one appearing in \eqref{eq:mildsolutionSPDE}. 

The first approach, the one we are going to deal with in the present paper, consists in associating a Brownian motion, with values in a Hilbert space, with the random noise. One can then define the stochastic integral as an infinite sum of It\^{o} integrals with respect to one-dimensional Brownian motions. This leads to solutions involving random 
functions taking values in suitable functional spaces, using semigroup theory (see, e.g., \cite{dapratozabczyk} for a detailed treatment). 
To our best knowledge, the most general result of existence and uniqueness of a function-valued solution to hyperbolic SPDEs is given in \cite{peszat}, where the author considers a semilinear stochastic wave equation having a uniformly elliptic second order operator $A$ in place of the Laplacian, with uniformly bounded coefficients depending on $x\in\R^d$, $d\geq 1$. There, sufficient conditions on the stochastic term $\dot\Xi$ and on the coefficients of $A$ are given, in order to find a unique function-valued solution.
In the present paper we show existence and uniqueness of a function-valued solution to a wider class of semilinear hyperbolic SPDEs, with possibly unbounded coefficients depending on $(t,x)\in[0,T]\times\R^d$, $d\geq 1$.

A disadvantage of the {\em function(al-spaces)-valued solution} $u$ of \eqref{eq:SPDE} sketched above is that,
in general, it cannot be evaluated in the spatial argument (usually, it is a random element in the $t$ (that is, \textit{time}) parameter, taking values in a $L^p(\R^d)$-modeled  Hilbert or Banach space).
Then, an alternative approach focuses instead on the concept of stochastic integral with respect to a martingale measure. That is, the stochastic integral in \eqref{eq:mildsolutionSPDE} is defined through the martingale measure derived from the random noise $\dot \Xi$. Here one obtains a 
so-called {\em random-field solution}, that is, $u$ is defined as a map associating a random variable to each $(t,x)\in[0,T_0]\times\Rd$, where $T_0>0$ is the time horizon of the equation. For more details, see, e.g., the classical references \cite{conusdalang,dalang,walsh}, and the recent papers \cite{ACSlinear,alessiandre}, where
the existence of a random-field solution in the case of linear hyperbolic SPDEs with $(t,x)$-dependent coefficients has been shown. 
On the other hand, a disadvantage of such {\em random-field solution} $u$ of \eqref{eq:SPDE} is that its construction for non-linear equations is based on the stationarity condition $\Lambda=\Lambda(t-s,x-y)$, which is fulfilled by SPDEs with constant coefficients, but cannot be assumed if we want to deal with more general linear operators $L$ in \eqref{eq:SPDE},
indeed admitting variable coefficients. 
Namely, we have shown in \cite{ACSlinear}
how random-field solutions can be constructed for arbitrary order,
linear (weakly) hyperbolic SPDEs, with (possibly unbounded)
coefficients smoothly depending on $(t,x)\in[0,T]\times\R^d$.

It is remarkable that, in many cases, the theory of integration with respect to processes taking values in functional spaces, and the theory of integration with respect to martingale measures, lead to the same solution $u$ (in some sense) of an SPDE, see \cite{dalangquer} for a precise comparison. We conclude the paper showing a result of that kind, comparing 
the function-valued solutions to \eqref{eq:SPDE} obtained in the present paper, in the special case of the linear equations, with the random-field solutions of the same equation found in \cite{ACSlinear}.

\medskip
We remark that in the present paper, as well as in 
\cite{ACSlinear,alessiandre}, the main tools used to construct and study
the solutions, namely, pseudodifferential and Fourier integral operators, come from microlocal analysis. To our best knowledge, in \cite{alessiandre} their full potential has been rigorously applied for the first time
within the theory of random-field solutions to hyperbolic SPDEs. Other applications of these operators in the context of S(P)DEs can be found in \cite{tindel}, where S(P)DEs are investigated in the framework of function-valued solutions by means of pseudodifferential operators, and in \cite{ObSc14}, where a program for employing Fourier integral operators in stochastic structural analysis is described. We are not aware of any
other systematic application of microlocal and Fourier integral operators
techniques. In particular, concerning the analysis of (weakly) semilinear 
hyperbolic SPDEs with unbounded coefficients, we provide it here. 
As it is customary for the classes of the associated deterministic PDEs,
we are interested in both the smoothness, as well as the decay 
at spatial infinity, of the solutions. Here we prove
an analog of such \textit{global regularity} properties, employing
suitable \textit{weighted Sobolev spaces}, namely, the so-called
Sobolev-Kato spaces.

\subsection{The equations we consider}\label{subsec:hyp}
As mentioned above, we study semilinear SPDEs \eqref{eq:SPDE} whose partial differential operators $L$ have coefficients in 
$(t,x)\in [0,T]\times\R^d$ that may admit a polynomial growth as $|x|\to\infty$. Namely, we treat \textit{hyperbolic equations} of arbitrary order $m\in\N$ of the form \eqref{eq:SPDE}, whose coefficients are defined on the whole space $\Rd$, with
\beqs\label{elle}
L=D_t^m-\displaystyle\sum_{j=1}^m A_j(t,x,D_x)D_t^{m-j},\qquad A_j(t,x,D)=\displaystyle\sum_{|\alpha|\leq j}a_{\alpha j}(t,x)D_x^\alpha,
\eeqs
where $m\geq 1$, $a_{\alpha j}\in C^\infty([0,T], C^\infty(\R^d))$ for $|\alpha|\le j$, $j=0,\dots,m$, and, 
for all $k\in\N_0$, $\beta\in\N_0^d$, there exists a constant $C_{jk\alpha \beta}>0$ such that
\[
	|\partial^k_t\partial^\beta_x a_{\alpha j}(t,x)|\le C_{jk\alpha \beta} \x^{|\alpha|-|\beta|},
\]
for all $(t,x)\in[0,T]\times\R^d$ and $0\leq|\alpha|\leq j$, $1\leq j\leq m$. 
The hyperbolicity of $L$ means that the symbol $\caL_m(t,x,\tau,\xi)$ of the $SG$-principal part of $L$, defined here below, satisfies
\begin{equation}\label{roots}
	\caL_m(t,x,\tau,\xi):= \tau^m-\displaystyle\sum_{j=1}^m\displaystyle\sum_{|\alpha|= j}{a}_{\alpha j}(t,x)\xi^\alpha\tau^{m-j}
	=\prod_{j=1}^m\left(\tau-\tau_j(t,x,\xi)\right),
\end{equation}
with $\tau_j(t,x,\xi)$ real-valued, $\tau_j\in C^\infty([0,T];S^{1,1}(\R^d))$, $j=1,\dots,m$. The latter 
means that, for any $\alpha,\beta\in\N_0^d$, $k\in\N_0$, there exists a constant $C_{jk\alpha\beta}>0$ such that
\[
	|\partial_t^k\partial^\alpha_x\partial^\beta_\xi\tau_j(t,x,\xi)|\le C_{jk\alpha\beta}\x^{1-|\alpha|}\csi^{1-|\beta|},
\]
for $(t,x,\xi)\in[0,T]\times\R^{2d}$, $j=1,\dots,m$ (see Section \ref{sec:fio} below for the definition of the so-called 
$SG$-classes of symbols $S^{m,\mu}(\R^d)$, $(m,\mu)\in\R^2$, and the corresponding class of pseudodifferential operators). 
The real solutions $\tau_j=\tau_j(t,x,\xi)$, $j=1,\dots,m$, of the equation 
$\caL_m(t,x,\tau,\xi)=0$ with respect to $\tau$ are usually called \textit{characteristic roots} of the operator $L$. We will
treat hyperbolic operators of three
different types:
\begin{itemize}
\item[(1)] strictly hyperbolic,
\item [(2)] weakly hyperbolic with roots of constant multiplicities,
\item [(3)] weakly hyperbolic with involutive roots,
\end{itemize}
where $(1)\subset(2)\subset(3)$. Postponing to Definition \ref{3types} in Section \ref{sec:fio} below their precise characterization, 
we give here examples of each one of them.
\begin{example} A simple example of a \textit{strictly hyperbolic operator $L$} is given by the so-called \textit{SG-wave operator}
\[L=D_t^2-\x^2 \langle D\rangle^2=D_t^2-(1+|x|^2)(1-\Delta_x)=-\partial_t^2+(1+|x|^2)\Delta_x-(1+|x|^2),\qquad x\in\R^d,\] 
having symbol $L(x,\tau,\xi)=L_2(x,\tau,\xi)=\tau^2-\x^2\csi^2$ and roots $\tau_{\pm}(x,\xi)=\pm\x\csi$, which are real, \textit{distinct}, and \textit{separated} at every point of $[0,T]\times\R^{2d}$, as well as \textit{at infinity}. 
Notice that $L=-\square_g-(1+|x|^2)$ is a wave operator, associated with a fixed, 
suitable Riemannian metric $g$ on $\R^d$, perturbed by a polynomially growing potential, cfr. \cite{CJT4}.
\end{example}
\begin{example}
An example of a \textit{weakly hyperbolic operator $L$ with roots of constant multiplicities} is given by
\[
	L=(D_t^2-\x^2\langle D\rangle^2)^2=D_t^4-2\x^2\langle D\rangle^2 D_t^2+\x^4\langle D\rangle^4+\Op(p),
	\qquad x\in\R^d,
\]
$p\in S^{3,3}(\R^d)$, where, for $c\in S^{m,\mu}(\R^d)$, $\Op(c)$ denotes the pseudodifferential
operator with symbol $c$, see Section \ref{sec:fio}. The $SG$-principal symbol of $L$ is here
$L_4(x,\tau,\xi)=(\tau^2-\x^2\csi^2)^2$, with \textit{separated} roots $\tau_{\pm}(x,\xi)=\pm\x\csi$, both of \textit{multiplicity 2}.
\end{example}
\begin{example}\label{ex:1.3}
An example of a \textit{weakly hyperbolic operator $L$ with involutive roots of non-constant multiplicities} is given by
\[
	L=(D_t+tD_{x_1}+D_{x_2})(D_t-(t-2x_2)D_{x_1}), \qquad x\in\R^2,
\]
see \cite{ACSlinear,Morimoto:1}.
\end{example}

\subsection{The results we get}
We consider the SPDE \eqref{eq:SPDE}
with $L$ as in \eqref{elle}, \eqref{roots} and $\Xi$ an 
$\mathcal S'(\R^d)$-valued Gaussian process with correlation measure $\Gamma$ and spectral measure $\mu$, see Subsection \ref{subnoise} for a precise definition. We derive conditions on the coefficients of $L$, on the right-hand side terms $\gamma$ and $\sigma$, and on the spectral measure $\mu$ 
(hence, on $\Xi$), such that there exists a unique function-valued (mild) solution to the corresponding Cauchy problem. 
Namely, we are going to prove that
\begin{enumerate}[label=\bfseries (H\arabic{enumi}),ref=\bfseries (H\arabic{enumi})]
	\item\label{hyp:st} if $L$ is strictly hyperbolic, and 
	$\ds\sup_{\eta\in\R^d}\ds\int_{\R^d}\frac{1}{(1+|\xi+\eta|^2)^{m-1}}\mu(d\xi)<\infty$, or
	\item\label{hyp:cm} if $L$ is weakly hyperbolic with constant multiplicities, and  
	$\ds\sup_{\eta\in\R^d}\int_{\R^d}\frac{1}{(1+|\xi+\eta|^2)^{m-l}}\mu(d\xi)<\infty$, 
	where $l$ is the maximum multiplicity of the characteristic roots,
\end{enumerate}
then, under suitable assumptions on $\gamma, \sigma$ and the Cauchy data, there exists a unique function-valued solution to \eqref{eq:SPDE}. Notice that the more general are the assumptions on $L$ (i.e., the larger is $l$), the smallest is the class of the stochastic noises that we can allow to get a function-valued solution. 

The main result of the paper is Theorem \ref{thm:semilinearweak} below, which is stated directly under the more general assumption \ref{hyp:cm}.
This is done in order to keep the paper within a reasonable length, recalling that
assumption \ref{hyp:st} is equivalent to \ref{hyp:cm} with $l=1$.
Theorem \ref{thm:semilinearweak} extends the results of \cite{peszat} to the case of general higher order hyperbolic equations with coefficients in $(t,x)$, not uniformly bounded with respect to $x$ and with roots that may coincide.

We will also formulate an expected result concerning the case of involutive roots, namely:
\begin{enumerate}[label=\bfseries (H\arabic{enumi}),ref=\bfseries (H\arabic{enumi})]
	\setcounter{enumi}{2}
	\item\label{hyp:in} if $L$ is weakly hyperbolic with involutive roots and
	$\displaystyle\int_{\R^d}\mu(d\xi)<\infty$,
\end{enumerate}
then, under suitable assumptions on $\gamma, \sigma$ and the Cauchy data, there exists a unique function-valued solution to \eqref{eq:SPDE}. Notice that the condition to be satisfied in \ref{hyp:in} corresponds to the \textit{limit case} $l=m$ of 
\ref{hyp:cm}. Notice also that all the three conditions coincide when $m=1$.

Here below we give two examples of diffusion coefficients $\sigma$ that we can allow, postponing to Section \ref{sec:nonlin} the precise description of the assumptions on $\sigma$.

\begin{example}
	Let $\sigma(t, x, u)=u^2$. Then, $\sigma$ is an admissible non-linearity
	for the equations we consider. More generally, we allow
	$\sigma(t, x, u)=u^n$, $n\in\N$, $n>2$.
\end{example}

\begin{example}
	A right-hand side explicitly depending on $(t,x)\in[0,T]\times\R^d$ 
	and $u$, which
	is admissible for the equations we consider, is
	\beqs\label{expower}
		\sigma(t,x,u)=\langle x\rangle^{l-m}\cdot\widetilde{\sigma}(t,u),
	\eeqs
	where $l$ is the maximum multiplicity of the roots and $\widetilde{\sigma}$ is regular in time, satisfies suitable
	mapping properties with respect to the Sobolev-Kato spaces,
	and is (uniformly, locally) Lipschitz-continuous with
	respect to the second variable,
	see Definition \ref{def:lip} and Example \ref{ex:sigmanonlin}
	below for the precise conditions.
	
	To our best knowledge, a diffusion coefficient of the rather
	general form \eqref{expower} has never been sistematically treated
	in the literature,
	except in \cite{sanzvuillermot}, where, for $m=2$,
	it has been incorporated in a certain model equation by means
	of ad-hoc techniques.
\end{example}

\begin{example}
	More generally, a routine extension of
	the theory developed in the present paper allows for a stochastic 
	term of the very general form 
	$$\sigma(t,x,u,D_xu,\ldots,D_x^\alpha u),\qquad|\alpha|\leq m-1$$
	in the right-hand side of \eqref{eq:SPDE}. The only difference consists
	in the form of the lipschitzianity assumptions and the corresponding 
	mapping properties, see Section 
	\ref{sec:nonlin} below.
\end{example}

\subsection{Tools we employ}
The main tools for proving existence and uniqueness of solutions to \eqref{eq:SPDE} will be Fourier integral operators with symbols in the so-called $SG$ classes. Such symbols classes have been introduced in the '70s by H.O. Cordes (see, e.g. \cite{cordes}) and C. Parenti \cite{PA72} (see also R. Melrose \cite{ME}). To construct the fundamental solution of \eqref{eq:SPDE} we will need, on one hand, to perform compositions between pseudo-differential operators and Fourier integral operators of $SG$ type, using the theory developed in \cite{coriasco99}, and, on the other hand, compositions between Fourier integral operators of $SG$ type with possibly different phase functions. The latter can be achieved using the composition results recently obtained in \cite{AleSandro}, with the aim of applying them in the present paper. The paper \cite{AleSandro} is quite technical, so here we will only recall and make use of the main composition theorems coming from the theory developed there.
 
The proof of the main Theorem \ref{thm:semilinearweak} of this paper follows an approach similar to the one adopted for the applications treated in 
\cite{AleSandro,Coriasco:998.2,Coriasco:998.3}. 
This consists in the reduction of equation \eqref{eq:SPDE} to a corresponding first order system, 
by an appropriate change of the unknown, then in the construction of the fundamental solution 
for the system, subsequently in coming back to the (formal) solution of the original equation \eqref{eq:SPDE}, and finally in the application of a fixed point scheme in suitable functional spaces. The associated system inherits the regularity of the coefficients, so, in the situation examined in the present paper, 
at worst we are going to obtain hyperbolic first order systems with distinct and separated eigenvalues, of constant multiplicities, or
involutive, cf. \cite{alessiandre,kumano-go}.

\subsection{Organization of the paper}
To provide a presentation of our results as
self-contained as possible, for the convenience 
of the reader, we provide (at different levels of detail)
various preliminaries from the existing literature, 
as described below.

In Section \ref{sec:stochastics} we recall some notions about stochastic integration with respect to Hilbert space-valued processes and the corresponding  concept of function-valued solution, following \cite{dapratozabczyk}.

In Section \ref{sec:fio} we give a description of the tools coming from microlocal analysis that we use for the construction of the fundamental solution to an hyperbolic first order system with polynomially bounded coefficients.  
The results presented in this section come mainly from \cite{ACSlinear, AleSandro,cordes,coriasco99,Coriasco:998.2,Coriasco:998.3}. 
We give a summary of them here, for the convenience of the reader.

In Section \ref{sec:nonlin} we focus on the semilinear hyperbolic SPDE \eqref{eq:SPDE}, \eqref{elle}, \eqref{roots}, and in Theorem \ref{thm:semilinearweak} we study existence and uniqueness of a function-valued solution under the assumption of weak hyperbolicity with roots of constant multiplicity \eqref{def:constmult}. The case of strict hyperbolicity  \eqref{def:strict} reduces to the special case $l=1$ of Theorem \ref{thm:semilinearweak}.
Moreover we recall the construction of the equivalent system performed in \cite{ACSlinear,AleSandro,Coriasco:998.2,Coriasco:998.3} and of its fundamental solution. The latter are crucial results, since all the three classes of hyperbolic equations we are going to consider can be reduced to a first order hyperbolic system.
We give sufficient conditions on the coefficients, on the noise and on the right-hand side of \eqref{eq:SPDE} such that there exists a unique mild function-valued solution of the corresponding Cauchy problem. The key result to achieve existence and uniqueness of the solution is Lemma \ref{lem:weightedpesz}, which is a further main result in the present paper.
We also state and comment a theorem concerning  similar results under the assumption of weak hyperbolicity with involutive roots \eqref{def:involutive}. Finally, we make
a comparison between the function-valued solutions obtained here, in the
special case of linear equations, with the random-field solutions found
in \cite{ACSlinear}, showing that they actually coincide.

\subsection{Notation}
Throughout this article, we let $\langle a\rangle:=(1+|a|^2)^{1/2}$ for all $a\in\Rd$, and we denote $\N_0:=\N\cup\{0\}$, $\R^d_*:=\Rd\backslash\{0\}$. Also, $\alpha$ and $\beta$ will generally denote multiindeces, with their standard arithmetic o\-pe\-ra\-tions. As usual, we will denote partial derivatives with $\partial$, and set $D=-\ii\partial$, $\ii$ being the imaginary unit, which is convenient when dealing with Fourier transformations. We will denote by $C^m(X)$, 
$C^m_0(X)$, $\caS(X)$, $\caD(X)$, $\caS'(X)$ and $\caD'(X)$, the $m$-times continuously differentiable functions, the $m$-times continuously differentiable functions with compact support, 
the Schwartz functions, the test functions space $C_0^\infty(X)$, the tempered distributions
and the distributions on some finite or infinite-dimensional space $X$, respectively. 
Usually, $C>0$ will denote a generic constant, whose value can change from line to line without further notice. 
When operator composition is considered, we will usually insert the symbol $\circ$ when the notation $\Op(b)$ and/or $\Op_\varphi(a)$,
for pseudodifferential and Fourier integral operators, respectively, are adopted for both factors, as well as in some situations where 
parameter-dependent operators occurs, for the sake of clarity. When at least one of the 
operators involved in the product of composition is denoted by a single capital letter, and when no confusion can occur, 
we will, as customary, omit the symbol $\circ$ completely, and just write, e.g., $PQ$, $RD_t$, etc. 
Finally, $A\asymp B$ means that the estimates $A\lesssim B$ and $B\lesssim A$ hold true,
where $A\lesssim B$ means that $|A|\le c\cdot |B|$, for a suitable constant $c>0$.

\section*{Acknowedgements}
The authors have been supported by the INdAM-GNAMPA grant 2014
``Equazioni Differenziali a Derivate Parziali di Evoluzione e Stocastiche'' (Coordinator: S. Coriasco, Dep. of Mathematics ``G. Peano'', University of Turin) and by the INdAM-GNAMPA grant 2015 ``Equazioni Differenziali a Derivate Parziali di Evoluzione e Stocastiche'' (Coordinator: A. Ascanelli, Dep. of Mathematics and Computer Science, University of Ferrara). The third author has been partially supported by the grant MTM 2015-65092-P by the Secretaria de estado de investigaci\'on, desarrollo e innovaci\'on, Ministerio de Econom\'ia y Competitividad, Espa$\rm{\tilde n}$a.

Thanks are due, for very useful discussions and observations,
to 
Tobias Hartung, 
Michael Oberguggenberger, Stevan Pilipovi\'c, Enrico Priola, 
Dora Sele\v{s}i, and Ingo Witt.

%
\section{Stochastic integration.}\label{sec:stochastics}
The mild formulation \eqref{eq:mildsolutionSPDE} is the way in which we understand the SPDE \eqref{eq:SPDE}.
In fact, we call \textit{(mild) function-valued solution to \eqref{eq:SPDE}} an $L^2(\Omega)$-family of random variables $u(t,x)$, $(t,x)\in[0,T]\times\R^d$, jointly measurable, satisfying the stochastic integral equation \eqref{eq:mildsolutionSPDE} where the last term in
the right-hand side is understood within the theory of
stochastic integrals taking value in Hilbert spaces, which we briefly now recall.

\subsection{The stochastic noise}\label{subnoise}
Here we describe the class of stochastic noises that we allow in our framework.
Consider a distribution-valued 
Gaussian process $\{\Xi(\phi);\; \phi\in\mathcal{C}_0^\infty(\mathbb{R}_+\times\Rd)\}$ on a complete probability space $(\Omega, \scrF, \P)$,
with mean zero and covariance functional given by
\begin{equation}
	\E[\Xi(\phi)\Xi(\psi)] = \int_0^\infty\int_\Rd \big(\phi(t)\ast\tilde{\psi}(t)\big)(x)\,\Gamma(dx) dt,
	\label{eq:correlation}
\end{equation}
where $\widetilde{\psi}(t,x) := \psi(t,-x)$, $\ast$ is the convolution operator and $\Gamma$ is a nonnegative, nonnegative definite, tempered measure on $\Rd$.
Then, Th\'{e}or\`{e}me XVIII in \cite[Chapter VII]{schwartz} implies that there exists a nonnegative tempered measure $\mu$ on $\Rd$ such that $\caF\mu = \widehat{\mu}=\Gamma$. $\caF$ and $\widehat{\phantom{\mu}}$ denote the Fourier transform given, for functions $f\in L^1(\Rd)$, by
\begin{equation}\label{eq:definitionfouriertransform}
	(\caF f)(\xi) = \widehat{f}(\xi) := \int_\Rd \e^{-\ii x\cdot\xi}f(x)dx.
\end{equation}
In \eqref{eq:definitionfouriertransform}, $x\cdot\xi$ denotes the inner product in $\Rd$, and the Fourier transform is extended to tempered distributions $T\in\caS'(\Rd)$ by the relation
$\langle \caF T,\phi\rangle = \langle T, \caF\phi\rangle,$ for all $\phi\in\caS(\Rd)$. By Parseval's identity, the right-hand side of \eqref{eq:correlation} can be rewritten as
\begin{equation*}
	\E[\Xi(\phi)\Xi(\psi)] = \int_0^{\infty}\int_{\Rd}[\caF\phi(t)](\xi)\
	\cdot
	\overline{[\caF\psi(t)](\xi)}\,\mu(d\xi) dt.
\end{equation*}
The tempered measure $\Gamma$ is usually called \emph{correlation measure}. The tempered measure $\mu$ such that $\Gamma=\widehat\mu$ is usually called \emph{spectral measure}.


\subsection{Stochastic integral in Hilbert spaces}\label{subH}
In this subsection we recall some of the main results of the theory of stochastic integration with respect to cylindrical Wiener processes. Also, we recall the definition of the Hilbert space $\mathcal H$ which will be suitable for our purposes of function-valued solutions to SPDEs. For the latter, we follow the exposition in \cite{dalangquer}.

\begin{definition}\label{cWp} 
  Let $Q$ be a self-adjoint, nonnegative definite and bounded linear operator on a separable Hilbert space $H$. An $H$-valued stochastic process $W = (W_t(h); h\in H, t\geq0)$ is called a {\em cylindrical Wiener process on $H$} on the complete probability space $(\Omega,\scrF,\P)$ if the following conditions are fulfilled:
  \begin{enumerate}
    \item for any $h\in H$, $(W_t(h); t\geq0)$ is a one-dimensional Brownian motion with variance $t\langle Qh,h\rangle_H$;
    \item for all $s,t\geq0$ and $g,h\in H$,
    \[ \E[W_s(g)W_t(h)] = (s\wedge t)\langle Qg,h\rangle_H. \]
  \end{enumerate}
  If $Q=Id_H$, then $W$ is called the standard cylindrical Wiener process.
\end{definition}

Let $\scrF_t$ be the $\sigma$-field generated by the random variables $(W_t(h); 0\leq s\leq t, h\in H)$ and the $\P$-null sets. The predictable $\sigma$-field is then the $\sigma$-field in $[0,T]\times\Omega$ generated by the sets $\{(s,t]\times A, A\in\scrF_t, 0\leq s<t\leq T\}$.

We define $H_Q$ to be the completion of the Hilbert space $H$ endowed with the inner product
\[ \langle g,h\rangle_{H_Q} := \langle Qg,h\rangle_H, \]
for $g,h\in H$. In the sequel, we let $(v_k)_{k\in\N}$ be a complete orthonormal basis of $H_Q$. Then, the stochastic integral of a predictable, square-integrable stochastic process with values in $H_Q$, $u\in L^2([0,T]\times\Omega; H_Q)$, is defined as
\[ \int_0^t u(s)dW_s := \sum_{k\in\N} \langle u,v_k\rangle_{H_Q} dW_s(v_k). \]
In fact, the series in the right-hand side converges in $L^2(\Omega,\scrF,\P)$ and its sum does not depend on the chosen orthonormal system
$(v_k)_{k\in\N}$. Moreover, the It\^o isometry
\[ \E\bigg[\bigg(\int_0^t u(s)dW_s\bigg)^2\bigg] = \E\bigg[\int_0^t \|u(s)\|_{H_Q}^2 ds\bigg] \]
holds true for any $u\in L^2([0,T]\times\Omega;H_Q)$.
For more on one-dimensional integration, see, e.g., \cite{oksendal}.

This notion of stochastic integral can also be extended to operator-valued integrands. Let $U$ be a separable Hilbert space and define $L_2^0 := L_2(H_Q,U)$ the set of Hilbert-Schmidt operators from $H_Q$ to $U$. With this we can define the space of integrable processes (with respect to $W$) as the set of $\scrF$-measureable processes in $L^2([0,T]\times\Omega;L_2^0)$. Since one can identify the Hilbert-Schmidt operators $L_2(H_Q,U)$ with $U\otimes H_Q^*$, one can define the stochastic integral for any $u\in L^2([0,T]\times\Omega;L_2^0)$ coordinatewise in $U$. Moreover, it is possible to establish an It\^o isometry, namely,
\beqs\label{isomhilb}
\E\Bigg[\bigg\|\int_0^t u(s)dW_s\bigg\|_U^2\Bigg] := \int_0^t \E\big[\|u(s)\|_{L_2^0}^2\big] ds.
\eeqs

The stochastic noise introduced in Subsection \ref{subnoise} can be rewritten in terms of a cylindrical Wiener process. The space $\mathcal{C}_0^\infty(\Rd)$, with pre-inner product 
\[ \langle\phi,\psi\rangle_{\caH} = \int_{\Rd}\caF\phi(\xi)\overline{\caF\psi}(\xi)\mu(d\xi),\]
can be completed to
\[ \caH := \overline{\mathcal{C}_0^\infty(\Rd)}^{\langle\cdot,\cdot\rangle_\caH}, \]
see \cite[Lemma 2.4]{dalangquer}. Then, $(\caH; \langle\cdot,\cdot\rangle_{\caH})$ is a real separable Hilbert space. We also set
\[ \caH_T := L^2([0,T];\caH). \] 
Then, \cite[Proposition 2.5]{dalangquer} states the following result.

\begin{proposition}
  For $t\geq0$ and $\phi\in\caH$, set $W_t(\phi) = W(1_{[0,t]}(\cdot)\phi( \cdot))$. Then, the process $W = \{W_t(\phi), t\geq0,\phi\in\caH\}$ is a standard cylindrical Wiener process on $\caH$ (where we
  recall that ``standard'' here means assuming $Q=Id_\caH$).
\end{proposition}

%
\section{Microlocal analysis and fundamental solution to first order hyperbolic systems with polynomially bounded coefficients}\label{sec:fio}

We first recall some basic definitions and facts about the so-called $SG$-calculus of pseudodifferential and Fourier integral operators, through
standard material appeared, e.g., in \cite{AleSandro} and elsewhere (sometimes with slightly different notational choices).

The class $S ^{m,\mu}=S ^{m,\mu}(\R^{d})$ of $SG$ symbols of order $(m,\mu) \in \R^2$ is given by all the functions 
$a(x,\xi) \in C^\infty(\R^d\times\R^d)$
with the property
that, for any multiindices $\alpha,\beta \in \N_0^d$, there exist
constants $C_{\alpha\beta}>0$ such that the conditions 
\begin{equation}
	\label{eq:disSG}
	|D_x^{\alpha} D_\xi^{\beta} a(x, \xi)| \leq C_{\alpha\beta} 
	\x^{m-|\alpha|}\csi^{\mu-|\beta|},
	\qquad (x, \xi) \in \R^d \times \R^d,
\end{equation}
hold, see, e.g., \cite{cordes,ME,PA72} for details. For $m,\mu\in\R$, $\ell\in\N_0$, $a\in\ S^{m,\mu}$, the quantities
\begin{equation}\label{seminorms}
	\vvvert a \vvvert^{m,\mu}_\ell
	= 
	\max_{|\alpha+\beta|\le \ell}\sup_{x,\xi\in\R^d}\x^{-m+|\alpha|} 
	                                                                     \csi^{-\mu+|\beta|}
	                                                                    | \partial^\alpha_x\partial^\beta_\xi a(x,\xi)|
\end{equation}
are a family of seminorms, defining  the Fr\'echet topology of $S^{m,\mu}$.

The corresponding
classes of pseudodifferential operators $\Op (S ^{m,\mu})=\Op (S ^{m,\mu}(\R^d))$ are given by
\begin{equation}\label{eq:psidos}
	(\Op(a)u)(x)=(a(.,D)u)(x)=(2\pi)^{-d}\int e^{\ii x\xi}a(x,\xi)\hat{u}(\xi)d\xi, \quad a\in S^{m,\mu}(\R^d),u\in\caS(\R^d),
\end{equation}
extended by duality to $\caS^\prime(\R^d)$.
The operators in \eqref{eq:psidos} form a
graded algebra with respect to composition, i.e.,
$$
\Op (S ^{m_1,\mu _1})\circ \Op (S ^{m_2,\mu _2})
\subseteq \Op (S ^{m_1+m_2,\mu _1+\mu _2}).
$$
The symbol $c\in S ^{m_1+m_2,\mu _1+\mu _2}$ of the composed operator $\Op(a)\circ\Op(b)$,
$a\in S ^{m_1,\mu _1}$, $b\in S ^{m_2,\mu _2}$, admits the asymptotic expansion
\begin{equation}
	\label{eq:comp}
	c(x,\xi)\sim \sum_{\alpha}\frac{i^{|\alpha|}}{\alpha!}\,D^\alpha_\xi a(x,\xi)\, D^\alpha_x b(x,\xi),
\end{equation}
which implies that the symbol $c$ equals $a\cdot b$ modulo $S ^{m_1+m_2-1,\mu _1+\mu _2-1}$.

The residual elements of the calculus are operators with symbols in
\[
	 S ^{-\infty,-\infty}=S ^{-\infty,-\infty}(\R^{d})= \bigcap_{(m,\mu) \in \R^2} S ^{m,\mu} (\R^{d})
	 =\caS(\R^{2d}),
\]
that is, those having kernel in $\caS(\R^{2d})$, continuously
mapping  $\caS^\prime(\R^d)$ to $\caS(\R^d)$. For any $a\in S^{m,\mu}$, $(m,\mu)\in\R^2$,
$\Op(a)$ is a linear continuous operator from $\caS(\R^d)$ to itself, extending to a linear
continuous operator from $\caS^\prime(\R^d)$ to itself, and from
$H^{z,\zeta}(\R^d)$ to $H^{z-m,\zeta-\mu}(\R^d)$,
where $H^{z,\zeta}(\R^d)$,
$(z,\zeta) \in \R^2$, denotes the  Sobolev-Kato (or \textit{weighted Sobolev}) space
\begin{equation}\label{eq:skspace}
  	H^{z,\zeta}(\R^d)= \{u \in \caS^\prime(\R^{n}) \colon \|u\|_{z,\zeta}=
	\|{\jap}^z\pdd^\zeta u\|_{L^2}< \infty\},
\end{equation}
with the naturally induced Hilbert norm. When $z\ge z^\prime$ and $\zeta\ge\zeta^\prime$, the continuous embedding 
$H^{z,\zeta}\hookrightarrow H^{z^\prime,\zeta^\prime}$ holds true. It is compact when $z>z^\prime$ and $\zeta>\zeta^\prime$.
Since $H^{z,\zeta}=\jap^z\,H^{0,\zeta}=\jap^z\, H^\zeta$, with $H^\zeta$ the usual Sobolev space of order $\zeta\in\R$, we 
find $\zeta>k+\dfrac{d}{2} \Rightarrow H^{z,\zeta}\hookrightarrow C^k$, $k\in\N_0$.
\begin{remark} Notice that in \cite{peszat} the author uses the space 
\[L^2_\omega:=\{u\in\mathcal S'(\R^d)\vert\ \sqrt\omega u\in L^2(\R^d)\},\]
where $\omega(x)\in\mathcal S(\R^d)$ is a strictly positive even function such that for $|x|\geq 1$ we have $\omega(x)=e^{-|x|}.$ 
The weight $\omega$ can be substituted by $\omega(x)=\x^{-2z},$ $z>0$, with corresponding space 
\[
	L^2_\omega:=\{u\in\mathcal S'(\R^d)\vert\ \x^{-z} u\in L^2(\R^d)\},
\]
coinciding with $H^{-z,0}(\R^d)$ in the notation above. 
In Section \ref{sec:nonlin} we shall use the $H^{z,\zeta}(\R^d)$ spaces to get a function-valued solution to \eqref{eq:SPDE}. 
\end{remark}

One actually finds
\begin{equation}\label{eq:spdecomp}
	\bigcap_{z,\zeta\in\R}H^{z,\zeta}(\R^d)=H^{\infty,\infty}(\R^d)=\caS(\R^d),
	\quad
	\bigcup_{z,\zeta\in\R}H^{z,\zeta}(\R^d)=H^{-\infty,-\infty}(\R^d)=\caS^\prime(\R^d),
\end{equation}
as well as, for the space of \textit{rapidly decreasing distributions}, see \cite{acs,schwartz}, 
\begin{equation}\label{eq:rdd}
	\caS^\prime(\R^d)_\infty=\bigcap_{z\in\R}\bigcup_{\zeta\in\R}H^{z,\zeta}(\R^d).
\end{equation}
The continuity property of
the elements of $\Op(S^{m,\mu})$ on the scale of spaces $H^{z,\zeta}(\R^d)$, $(m,\mu),(z,\zeta)\in\R^2$, is expressed 
more precisely in the next Theorem \ref{thm:sobcont} (see \cite{cordes} and the references quoted therein 
for the result on more general classes of $SG$-symbols).
\begin{theorem}\label{thm:sobcont}
	Let $a\in S^{m,\mu}(\R^d)$, $(m,\mu)\in\R^2$. Then, for any $(z,\zeta)\in\R^2$, 
	$\Op(a)\in\mathcal{L}(H^{z,\zeta}(\R^d),H^{z-m,\zeta-\mu}(\R^d))$, and there exists a constant $C>0$,
	depending only on $d,m,\mu,z,\zeta$, such that
	\begin{equation}\label{eq:normsob}
		\|\Op(a)\|_{\scrL(H^{z,\zeta}(\R^d), H^{z-m,\zeta-\mu}(\R^d))}\le 
		C\vvvert a \vvvert_{\left[\frac{d}{2}\right]+1}^{m,\mu},
	\end{equation}
	where $[t]$ denotes the integer part of $t\in\R$.
\end{theorem}
Cordes introduced the class $\caO(m,\mu)$ of the \textit{operators of order $(m,\mu)$} as follows, see, e.g., \cite{cordes}.
\begin{definition}\label{def:ordmmuopr}
	A linear continuous operator $A\colon\caS(\R^d)\to\caS(\R^d)$
	belongs to the class $\caO(m,\mu)$, $(m,\mu)\in\R^2$, of the operators of order $(m,\mu)$ if, for any $(z,\zeta)\in\R^2$,
	it extends to a linear continuous operator $A_{z,\zeta}\colon H^{z,\zeta}(\R^d)\to H^{z-m,\zeta-\mu}(\R^d)$. We also define
	\[
		\caO(\infty,\infty)=\bigcup_{(m,\mu)\in\R^2} \caO(m,\mu), \quad
		\caO(-\infty,-\infty)=\bigcap_{(m,\mu)\in\R^2} \caO(m,\mu).		
	\]
\end{definition}
\begin{remark}\label{rem:O}
	\begin{enumerate}
		\item Trivially, any $A\in\caO(m,\mu)$ admits a linear continuous extension 
		$A_{\infty,\infty}\colon\caS^\prime(\R^d)\to\caS^\prime(\R^d)$. In fact, in view of \eqref{eq:spdecomp}, it is enough to set
		$A_{\infty,\infty}|_{H^{z,\zeta}(\R^d)}= A_{z,\zeta}$.
		\item Theorem \ref{thm:sobcont} implies $\Op(S^{m,\mu}(\R^d))\subset\caO(m,\mu)$, $(m,\mu)\in\R^2$.
		\item $\caO(\infty,\infty)$ and $\caO(0,0)$ are algebras under operator multiplication, $\caO(-\infty,-\infty)$ is an ideal
		of both  $\caO(\infty,\infty)$ and $\caO(0,0)$, and 
		$\caO(m_1,\mu_1)\circ\caO(m_2,\mu_2)\subset\caO(m_1+m_2,\mu_1+\mu_2)$.
	\end{enumerate}
\end{remark}
\noindent
The following characterization of the class $\caO(-\infty,-\infty)$ is often useful, see \cite{cordes}.
\begin{theorem}\label{thm:smoothing}
	The class $\caO(-\infty,-\infty)$ coincides with $\Op(S^{-\infty,-\infty}(\R^d))$ and with the class of smoothing operators,
	that is, the set of all the linear continuous operators $A\colon\caS^\prime(\R^d)\to\caS(\R^d)$. All of them coincide with the
	class of linear continuous operators $A$ admitting a Schwartz kernel $k_A$ belonging to $\caS(\R^{2d})$. 
\end{theorem}
An operator $A=\Op(a)$ and its symbol $a\in S ^{m,\mu}$ are called \emph{elliptic}
(or $S ^{m,\mu}$-\emph{elliptic}) if there exists $R\ge0$ such that
\[
	C\x^{m} \csi^{\mu}\le |a(x,\xi)|,\qquad 
	|x|+|\xi|\ge R,
\] 
for some constant $C>0$. If $R=0$, $a^{-1}$ is everywhere well-defined and smooth, and $a^{-1}\in S ^{-m,-\mu}$.
If $R>0$, then $a^{-1}$ can be extended to the whole of $\R^{2d}$ so that the extension $\widetilde{a}_{-1}$ satisfies $\widetilde{a}_{-1}\in S ^{-m,-\mu}$.
An elliptic $SG$ operator $A \in \Op (S ^{m,\mu})$ admits a
parametrix $A_{-1}\in \Op (S ^{-m,-\mu})$ such that
\[
A_{-1}A=I + R_1, \quad AA_{-1}= I+ R_2,
\]
for suitable $R_1, R_2\in\Op(S^{-\infty,-\infty})$, where $I$ denotes the identity operator. 
In such a case, $A$ turns out to be a Fredholm
operator on the scale of functional spaces $H^{z,\zeta}(\R^d)$,
$(z,\zeta)\in\R^2$.

\vspace{3mm}

We now introduce the class of $SG$-phase functions.
\begin{definition}[$SG$-phase function]\label{def:phase}
A real valued function $\varphi\in C^\infty(\R^{2d})$  belongs to the class $\Ph$ of $SG$-phase functions if it satisfies the following conditions:
\begin{enumerate}
\item $\varphi\in S^{1,1}(\R^{d})$;
\item $\<\varphi'_x(x,\xi)\>\asymp\<\xi\>$ as $|(x,\xi)|\to\infty$;
\item $\<\varphi'_\xi(x,\xi)\>\asymp\<x\>$ as $|(x,\xi)|\to\infty$.	
\end{enumerate}
\end{definition}
For any $a\in S^{m,\mu}$, $(m,\mu)\in\R^2$, $\varphi\in\Ph$, the $SG$ FIOs are defined, for $u\in\caS(\R^{n})$, as
\begin{align}
\label{eq:typei}
(\Op _\fy (a)u)(x)&= (2\pi )^{-d}\int
e^{\ii\fy (x,\xi )} a(x,\xi )
\widehat u(\xi )\, d\xi ,
\end{align}
and
\begin{align}
\label{eq:typeii}
(\Op^*_\fy (a)u)(x)&= (2\pi )^{-d}\iint
e^{\ii(x\cdot \xi -\fy (y,\xi ))} \overline {a(y,\xi )} u(y)\, dyd\xi.
\end{align}
Here the operators $\Op _\fy (a)$ and
$\Op _\fy ^*(a)$ are sometimes called $SG$ FIOs
of type I and type II, respectively, with symbol $a$ and
($SG$-)phase function $\fy$. Note that a type II operator satisfies
$\Op^*_\fy(a)=\Op _\fy (a)^*$, that is, it is the formal $L^2$-adjoint of the 
type I operator $\Op_\fy(a)$. 

The analysis of $SG$ FIOs started in 
\cite{coriasco99}, where composition results with the classes of $SG$ pseudodifferential operators,
and of $SG$ FIOs of type I and type II with regular phase functions, have been proved. Also the basic
continuity properties in $\caS(\R^d)$ and $\caS^\prime(\R^d)$ of operators in the class have been proved there, as well
as a version of the Asada-Fujiwara $L^2(\R^d)$-continuity, for operators $\Op_\fy(a)$ 
with symbol $a\in S^{0,0}$ and regular $SG$-phase function $\fy\in\Phr$, see Definition \ref{def:phaser}.
The following theorem summarizes composition results between $SG$ pseudodifferential operators and $SG$ FIOs of type I that we are going to use in the present paper, see \cite{coriasco99} for proofs and composition results with $SG$ FIOs of type II.
\begin{theorem}\label{thm:compi}
Let $\fy\in\Ph$ and assume $b\in S^{m_1,\mu_1}(\R^{d})$, $a\in S^{m_2,\mu_2}(\R^{d})$, $(m_j,\mu_j)\in\R^2$, $j=1,2$. Then,
\begin{align*}
\Op (b)\circ \Op _\fy (a) &= \Op _\fy (c_1+r_1) = \Op _\fy (c_1) \mod \Op (S^{-\infty,-\infty}(\R^{d})),
\\[1ex]
\Op _\fy (a) \circ \Op (b) &= \Op _\fy (c_2+r_2) = \Op _\fy (c_2)  \mod \Op (S^{-\infty,-\infty}(\R^{d})),
\end{align*}
for some $c_j\in S^{m_1+m_2,\mu_1+\mu_2}(\R^{d})$, $r_j\in S^{-\infty,-\infty}(\R^{d})$, $j=1,2$.
\end{theorem}
To consider the composition of $SG$ FIOs of type I and type II some more hypotheses are needed, leading to the definition
of the classes $\Phr$ and $\Phr(\lambda)$ of regular $SG$-phase functions. 
\begin{definition}[Regular $SG$-phase function]\label{def:phaser}
Let $\lambda\in [0,1)$ and $\delta>0$. A function $\varphi\in\Ph$ belongs to the class $\Phr(\lambda)$ if it satisfies the following conditions:
\begin{enumerate}
\item $\vert\det(\varphi''_{x\xi})(x,\xi)\vert\geq \delta$, $\forall (x,\xi)$;
\item the function  $J(x,\xi):=\varphi(x,\xi)-x\cdot\xi$ is such that
\beqs\label{hyp}
\ds\sup_{\afrac{x,\xi\in\R^d}{|\alpha+\beta|\leq 2}}\frac{|D_\xi^\alpha D_x^\beta J(x,\xi)|}{\x^{1-|\beta|}\<\xi\>^{1-|\alpha|}}\leq \lambda.
\eeqs
\end{enumerate}
If only condition (1) holds, we write $\fy\in\Phr$.
\end{definition}
\begin{remark}
Notice that condition \eqref{hyp} means that $J(x,\xi)/\lambda$ is bounded with constant $1$ in $S^{1,1}((2))$, the Fr\'echet space
of symbols $a\in C^\infty(\R^d\times\R^d)$ such that the estimates \eqref{eq:disSG} hold for $|\alpha+\beta|\le2$ only. 
Notice also that condition (1) in Definition \ref{def:phaser} is authomatically fulfilled when condition (2) holds true for a 
sufficiently small $\lambda\in[0,1)$.
\end{remark}
For $\ell\in\N_0$, we also introduce the seminorms
\[
	\|J\|_{2,\ell}:=
	\ds\sum_{2\leq |\alpha+\beta|\leq 2+\ell}\sup_{x,\xi\in \R^{d}}
	\ds\frac{|D_\xi^\alpha D_x^\beta J(x,\xi)|}{\x^{1-|\beta|}\<\xi\>^{1-|\alpha|}},
\]
and
\[
\|J\|_\ell:=\ds\sup_{\afrac{x,\xi\in\R^d}{|\alpha+\beta|\leq 1}}\frac{|D_\xi^\alpha D_x^\beta J(x,\xi)|}{\x^{1-|\beta|}\<\xi\>^{1-|\alpha|}}+\|J\|_{2,\ell},
\]
cfr. \cite{kumano-go}. We notice that $\varphi\in\Phr(\lambda)$ means that (1) of Definition \ref{def:phaser} and $\|J\|_0\leq \lambda$ hold, and then we define the following subclass of the class of regular $SG$ phase functions:
\begin{definition}\label{def:phaserell}
Let $\lambda\in [0,1)$, $\delta>0$, $\ell \geq 0$. A function $\varphi$ belongs to the class $\Phr(\lambda,\ell)$ if $\varphi\in\Phr(\lambda)$ and $\|J\|_\ell\leq \lambda$ for the corresponding $J$.
\end{definition}

The result of a composition of $SG$ FIOs of type I and type II with the same regular $SG$-phase functions is a $SG$ pseudodifferential operator, see again \cite{coriasco99}. The continuity properties of regular $SG$ FIOs on the Sobolev-Kato spaces can be expressed as
follows, using the operators of order $(m,\mu)\in\R^2$ introduced above. 
\begin{theorem}\label{thm:SGAF}
	Let $\varphi$ be a regular $SG$ phase function and $a\in S^{m,\mu}(\R^d)$, $(m,\mu)\in\R^2$. Then, $\Op_\varphi(a)\in\caO(m,\mu)$.
\end{theorem}

The study of the composition of $M\geq 2$ $SG$ FIOs of type I $\Op_{\varphi_j}(a_j)$ with regular $SG$-phase functions 
$\varphi_j\in\Phr(\lambda_j)$ and symbols $a_j\in S^{m_j,\mu_j}(\R^{d})$, $j=1,\ldots,M$,
has been done in \cite{AleSandro}. The result of such composition is still an SG-FIO with a regular SG-phase function $\varphi$ given by the so-called \textit{multi-product} $\varphi_1\sharp\cdots\sharp\varphi_M$ of the phase functions $\varphi_j$, $j=1,\ldots,M$, and symbol $a$ as in Theorem \ref{thm:mainprbis} here below.

\begin{theorem}\label{thm:mainprbis} 
Consider, for $j=1,2, \dots, M$, $M\ge 2$, the $SG$ FIOs of type I $\Op_{\fy_j}(a_j)$ with $a_j\in S^{m_j,\mu_j}(\R^{d})$, $(m_j,\mu_j)\in\R^2$, and $\fy_j\in\Phr(\lambda_j)$ such that $\lambda_1+\cdots+\lambda_M\le\lambda\le\frac{1}{4}$ for some sufficiently small $\lambda>0$.
Then, there exists $a\in S^{m,\mu}(\R^{d})$, $m=m_1+\cdots+m_M$, $\mu=\mu_1+\cdots+\mu_M$, such that, setting $\phi=\fy_1\sharp\cdots\sharp\fy_M$, we have
\[\Op_{\fy_1}(a_1) \circ\cdots\circ \Op_{\fy_M}(a_M)=\Op_{\phi}(a).\]
Moreover, for any $\ell\in\N_0$ there exist $\ell^\prime\in\N_0$, $C_\ell>0$ such that 
\begin{equation}\label{eq:estsna}
\vvvert a \vvvert_\ell^{m,\mu} \le C_\ell\prod_{j=1}^M \vvvert a_j \vvvert_{\ell^\prime}^{m_j,\mu_j}.
\end{equation}
\end{theorem}

Theorem \ref{thm:mainprbis} is a corollary of the main Theorem in \cite{AleSandro}. There, the \textit{multi-product} 
of regular $SG$-phase functions is defined and its properties are studied, parametrices and compositions of regular $SG$ FIOs with amplitude identically equal to $1$
are considered, leading to the general composition $\Op_{\fy_1}(a_1) \circ\cdots\circ \Op_{\fy_M}(a_M)$. In the present paper we only recall this composition result, since it is needed for the determination of the fundamental solutions of the hyperbolic operators \eqref{elle}, involved in \eqref{eq:SPDE}, in the case of involutive roots with non-constant multiplicities.

\vspace{3mm}

Applications of the $SG$ FIOs theory to $SG$-hyperbolic Cauchy problems were initially given in \cite{Coriasco:998.2, Coriasco:998.3}.
Many authors have, since then, expanded the $SG$ FIOs theory and its applications to the solution of hyperbolic problems
in various directions. To mention a few, see, e.g.,
M. Ruzhansky, M. Sugimoto \cite{RuSu}, 
E. Cordero, F. Nicola, L Rodino \cite{CorNicRod1}, and the references quoted there and in \cite{AleSandro}.

In \cite{AleSandro}, the results in Theorem \ref{thm:mainprbis} have been applied to study classes of $SG$-hyperbolic Cauchy problems, constructing their fundamental solution $\{E(t,s)\}_{0\leq s\leq t\leq T}$. The existence of the fundamental solution provides, via Duhamel's formula, existence and uniqueness of the solution to the system, for any given Cauchy data in the weighted Sobolev spaces $H^{z,\zeta}(\R^d)$, $(z,\zeta)\in\R^2$. 
A remarkable feature, typical for these classes of hyperbolic problems, is the
\textit{well-posedness with loss/gain of decay at infinity}, observed for the first time in \cite{AC06}, and then in \cite{AC10}.

\vspace{3mm}
In the present paper we will deal with the following three classes of equations of the form \eqref{eq:SPDE}, and corresponding operators $L$:
\begin{enumerate}
\item {\em strictly hyperbolic equations}, that is, $\caL_m$ satisfies
\eqref{roots} with real-valued, distinct and separated roots $\tau_j$, $j=1,\dots,m$, in the sense that there exists a constant $C>0$ such that 
\begin{equation}\label{def:strict}
|\tau_j(t,x,\xi)-\tau_k(t,x,\xi)|\geq C\x\csi,\quad \forall j\neq k,\ (t,x,\xi)\in[0,T]\times\R^{2d};
\end{equation}
\item {\em hyperbolic equations with (roots of) constant multiplicities}, that is, $\caL_m$ satisfies
\eqref{roots} and the real-valued, characteristic roots can be divided into $n$ groups ($1\leq n\leq m$) of distinct and separated roots, in the sense that, possibly
after a reordering of the $\tau_j$, $j=1,\dots, m$, there exist 
$l_1,\ldots l_n\in\N$ with $l_1+\ldots+l_n=m$ and $n$ sets
\[ G_1=\{\tau_1=\cdots=\tau_{l_1}\},\quad G_2=\{\tau_{l_1+1}=\cdots=\tau_{l_1+l_2}\},\quad \ldots \quad 
G_n=\{\tau_{m-l_n+1}=\cdots=\tau_{m}\},\]
satisfying, for a constant $C>0$,
\begin{equation}\label{def:constmult}\tau_j\in G_p,\tau_k\in G_q ,\ p\neq q,\ 1\leq p,q\leq n\Rightarrow |\tau_j(t,x,\xi)-\tau_k(t,x,\xi)|\geq C\x\csi,\quad \forall (t,x,\xi)\in[0,T]\times\R^{2d};\end{equation}
notice that, in the case $n=1$, we have only one group of $m$ coinciding roots, that is, $\caL_m$ 
admits a single real root of multiplicity $m$,
while for $n=m$ we recover the strictly hyperbolic case; the number $l=\max_{j=1,\dots,n}l_j$ is the \textit{maximum multiplicity of the roots of $\caL_m$};
\item {\em hyperbolic equations with involutive roots}, that is, $\caL_m$ satisfies \eqref{roots} with 
real-valued characteristic roots such that
\begin{eqnarray}\label{def:involutive}
[D_t-\Op(\tau_j(t)), D_t-\Op(\tau_k(t))]=&
&\Op(a_{jk}(t))\,(D_t-\Op(\tau_j(t))
\\
\nonumber
&+&\Op(b_{jk}(t))\,(D_t-\Op(\tau_k(t)))+\Op(c_{jk}(t)),
\end{eqnarray}
for some $a_{jk},b_{jk},c_{jk}\in C^\infty([0,T],S^{0,0}(\R^d))$, $j,k=1,\dots,m$. 
\end{enumerate} 
\begin{remark}
Recall that roots of constant multiplicities are always involutive, see,
e.g., \cite{ACSlinear} for a proof. The converse statement is not true in general, as shown in Example \ref{ex:1.3}. 
\end{remark}
\begin{definition}\label{3types}
	We will say that the (linear) operator $L$ in \eqref{elle} is
	\textit{strictly ($SG$-)hyperbolic}, \textit{weakly ($SG$-)hyperbolic with constant multiplicities}, or 
	\textit{weakly ($SG$-)hyperbolic with involutive roots}, respectively, if such properties are satisfied by the
	roots of $\caL_m$, as explained above.
\end{definition}

\begin{remark}The \textit{asymmetric} case of coefficients satisfying, for constants $C_{jk\alpha \beta}>0$,
$|\partial^k_t\partial^\beta_x a_{\alpha j}(t,x)|\le C_{jk \alpha \beta} \x^{m_\alpha-|\beta|}$ 
for all $(t,x)\in[0,T]\times\R^d$ and $m_\alpha\leq |\alpha|$, $|\alpha|\leq j$, $k\in\N_0$, 
$1\leq j\leq m$, is currently under investigation by the authors.
Indeed, in that situation we would obtain real characteristic roots $\tau_j\in C^\infty([0,T];S^{\varepsilon,1}(\R^d))$, with $\varepsilon:=\max\{m_\alpha/j\ \vert\ |\alpha|\le j, 1\leq j\leq m\}\in[0,1]$. The particular case $\varepsilon=0$ is already known, see \cite{alessiandre,peszat}, while the \textit{symmetric} case $\varepsilon =1$ is the one we focus on in the present paper. The analysis of the
general case $\epsilon\in(0,1)$ will appear elsewhere.
\end{remark}

The next one is a key result in the analysis of $SG$-hyperbolic Cauchy problems by means of the corresponding class of Fourier
operators.
Given a symbol $\varkappa\in C([0,T]; S^{1,1})$, set $\Delta_{T_0}=\{(s,t)\in[0,T_0]^2\colon 0\le s\le t\le T_0\}$, $0<T_0\le T$,
and consider the eikonal equation
\beqs\label{eik}
\begin{cases}
\partial_t\varphi(t,s,x,\xi)=\varkappa(t,x,\varphi'_x(t,s,x,\xi)),& t\in [s,T_0],
\\
\varphi(s,s,x,\xi)=x\cdot\xi,& s\in [0,T_0),
\end{cases}
\eeqs
with $0<T_0\leq T$. By an extension of the theory developed  in \cite{Coriasco:998.2}, it is possible to
prove that the following Proposition \ref{trovala!} holds true.
\begin{proposition}\label{trovala!}
For any small enough $T_0\in(0,T]$, 
equation \eqref{eik} admits a unique solution $\varphi\in C^1(\Delta_{T_0},$ $S^{1,1}(\R^d))$,
satisfying $J\in C^1(\Delta_{T_0},S^{1,1}(\R^d))$ and
\beqs\label{eiks}
\partial_s\varphi(t,s,x,\xi)=-\varkappa(s,\varphi'_\xi(t,s,x,\xi),\xi),
\eeqs
for any $(t,s)\in\Delta_{T_0}$. Moreover, 
for every $\ell\in\N_0$ there exists $\delta>0$, $c_\ell\geq 1$ and $\widetilde{T}_\ell\in[0,T_0]$ such that 
$\varphi(t,s,x,\xi)\in\Phr(c_\ell|t-s|)$, 
with $\| J\|_{2,\ell}\leq c_\ell |t-s|$ for all $(t,s)\in\Delta_{\widetilde{T}_\ell}$.
\end{proposition}
\begin{remark}
	Of course, if additional regularity with respect to $t\in[0,T]$
	is fulfilled by the symbol $\varkappa$ in the right-hand side of \eqref{eik}, 
	this reflects in a corresponding increased regularity of the resulting
	solution $\varphi$ with respect to $(t,s)\in\Delta_{T_0}$. Since here we are not dealing with
	problems concerning the $t$-regularity of the solution, we assume smooth $t$-dependence of
	the coefficients of $L$. Some of the results below will anyway be formulated in situations of lower
	regularity with respect to $t$.
\end{remark}

By the hyperbolicity hypotheses \eqref{def:strict}, \eqref{def:constmult}, or \eqref{def:involutive}, as it will be shown below, to obtain  
the term $v_0$ and the kernel $\Lambda$, associated with the operator in \eqref{eq:SPDE},
it is enough to know the fundamental solution of certain first order systems.
Namely, let us consider the Cauchy problem
\begin{equation}\label{sys}
\begin{cases}
	(D_t - \Op(\kappa_1(t)) - \Op(\kappa_0(t)))W(t) = Y(t), & t\in [0,T], \\
	W(s)  = W_0, & s\in [0,T],
\end{cases}
\end{equation}
where the $(\nu\times\nu)$-system is hyperbolic with diagonal principal part, that is:
\begin{itemize}
	\item[-] the matrix $\kappa_1$ satisfies $\kappa_1\in C^\infty([0,T],S^{1,1})$, it is real-valued and diagonal, 
	and each entry on the principal diagonal coincides with the value of one of the
	roots $\tau_j\in C^\infty([0,T]; S^{1,1})$, possibly repeated a number of times, depending on their multiplicities;
	\item[-] the matrix $\kappa_0$ satisfies $\kappa_0\in C^\infty([0,T],S^{0,0})$.
\end{itemize}
In analogy with the terminology introduced above, we will say that the system \eqref{sys} is strictly hyperbolic, or
hyperbolic with constant multiplicities, when the elements on the main diagonal of $\kappa_1$ are all distinct and
satisfy \eqref{def:strict}, or when they satisfy \eqref{def:constmult}, respectively. Similarly, we will say that the system
is hyperbolic with involutive roots when they satisfy \eqref{def:involutive}. We will also generally assume $W_0\in H^{z,\zeta}$,
$Y\in C([0,T],H^{z,\zeta})$, $(z,\zeta)\in\R^2$.

The fundamental solution, or \textit{solution operator},
of \eqref{sys} is a family $\{E(t,s)\colon (t,s)\in[0,T_0]^2\}$, $0<T_0\le T$, of linear continuous operators in the strong topology of
$\scrL(H^{z,\zeta}, H^{z,\zeta})$, $(z,\zeta)\in\R^2$. 
In the cases of strict $SG$-hyperbolicity or of $SG$-hyperbolicity with constant multiplicities,
such family can be explicitly expressed in terms of suitable (matrices of) 
$SG$ FIOs of type I, modulo smoothing terms, see \cite{Coriasco:998.2, Coriasco:998.3} and Section \ref{subs:reduction} below. 
In the case of $SG$-hyperbolicity with variable multiplicities,
it is, in general, a limit of a sequence of (matrices of) $SG$ FIOs of type I. In all three cases, it satisfies
\begin{equation}\label{tocheck}
	\begin{cases}
		(D_t - \Op (\kappa_1(t)) - \Op(\kappa_0(t)))E(t,s) = 0, 	& (t,s)\in [0,T_0]^2,\\
		E(s,s)=I,						& s\in[0,T_0].
	\end{cases}
\end{equation}
More precisely, $E(t,s)$ has the following properties, see \cite{ACSlinear}, which actually hold true for the broader class of symmetric first order system of the
type \eqref{sys}, of which systems with real-valued, diagonal principal part are a special case, 
see \cite{cordes}, Ch. 6, \S 3, and \cite{Coriasco:998.2}.
\begin{theorem}\label{thm:fundsolabs}
	Let the system \eqref{sys} be hyperbolic with diagonal principal part $\kappa_1\in C^1([0,T], S^{1,1}$ $(\R^d))$,
	and lower order part $\kappa_0\in C^1([0,T], S^{0,0}(\R^d))$. Then, for any choice of $W_0\in H^{z,\zeta}(\R^d)$,
	$Y\in C([0,T], H^{z,\zeta}(\R^d))$,
	there exists a unique solution $W\in C([0,T], H^{z,\zeta}(\R^d))\cap C^1([0,T], H^{z-1,\zeta-1}(\R^d))$ 
	of \eqref{sys}, $(z,\zeta)\in\R^2$, given by Duhamel's formula
	\[
		W(t)=E(t,s)W_0+i\ds\int_s^t E(t,\vartheta)Y(\vartheta)d\vartheta,\quad t\in[0,T].
	\]
	Moreover, the \textit{solution operator} $E(t,s)$ has the following properties:
	\begin{enumerate}
		\item $E(t,s)\colon \caS^\prime(\R^d)\to\caS^\prime(\R^d)$ is an operator belonging to $\caO(0,0)$, $(t,s)\in[0,T]^2$;
		its first order derivatives, $\partial_t E(t,s)$, $\partial_s E(t,s)$, exist in the strong operator convergence of 
		$\scrL(H^{z,\zeta}(\R^d),H^{z-1,\zeta-1}(\R^d))$, $(z,\zeta)\in\R^2$, and belong to $\caO(1,1)$;
		\item $E(t,s)$ is bounded and strongly continuous on $[0,T]^2_{ts}$ in
		$\scrL(H^{z,\zeta}(\R^d),H^{z,\zeta}(\R^d))$, $(z,\zeta)\in\R^2$;
		$\partial_t E(t,s)$ and $\partial_s E(t,s)$ are bounded and strongly continuous on $[0,T]^2_{ts}$ in
		$\scrL(H^{z,\zeta}(\R^d),H^{z-1,\zeta-1}(\R^d))$, $(z,\zeta)\in\R^2$;
		\item for $t,s,t_0\in[0,T]$ we have
		\[
			E(t_0,t_0)=I, \quad
			E(t,s)E(s,t_0)=E(t,t_0), \quad
			E(t,s)E(s,t)=I;
		\]
		\item $E(t,s)$ satisfies, for $(t,s)\in [0,T]^2$, the differential equations
		\begin{align}
			\label{eq:U1}
			D_tE(t,s) -  (\Op (\kappa_1(t)) + \Op(\kappa_0(t)))E(t,s) &= 0,
			\\
			\label{eq:U2}
			D_sE(t,s) + E(t,s)(\Op (\kappa_1(s)) + \Op(\kappa_0(s))) &= 0;
		\end{align}
		\item the operator family $E(t,s)$ is uniquely determined by the properties (1)-(3) here above, and one
		of the differential equations \eqref{eq:U1}, \eqref{eq:U2}.
	\end{enumerate} 
\end{theorem}
\begin{corollary}\label{cortocite}
	\begin{enumerate}
		\item Under the hypotheses of Theorem \ref{thm:fundsolabs}, $E(t,s)$ is invertible on $\caS(\R^d)$, $\caS^\prime(\R^d)$,
			and $H^{z,\zeta}(\R^d)$, $(z,\zeta)\in\R^2$, with inverse given by $E(s,t)$, $s,t\in[0,T]$. 
		\item If, additionally, one assumes $\kappa_1\in C^m([0,T], S^{1,1}(\R^d))$, $\kappa_0\in C^m([0,T], S^{0,0}(\R^d))$, $m\ge2$,
			the partial derivatives $\partial_t^j\partial_s^k E(t,s)$ exist in strong operator convergence of $\caS(\R^d)$
			and $\caS^\prime(\R^d)$, and $\partial_t^j\partial_s^k E(t,s)\in\caO(j+k,j+k)$, $j+k\le m$. Moreover, 
			$\partial_t^j\partial_s^k E(t,s)$ is strongly continuous on $[0,T]^2_{ts}$ in every
			$\scrL(H^{z,\zeta}(\R^d),$ $H^{z-j-k,\zeta-j-k}(\R^d))$, $(z,\zeta)\in\R^2$, $j+k\le m$.
	\end{enumerate}
\end{corollary}
In \cite{AleSandro} we have proved the next Theorem \ref{thm:fundsol}, concerning the structure of $E(t,s)$, in the spirit of 
the approach followed in \cite{kumano-go}.
\begin{theorem}\label{thm:fundsol}
Under the same hypotheses of Theorem \ref{thm:fundsolabs},
if $T_0$ is small enough, 
for every fixed $(t,s)\in\Delta_{T_0}$, 
$E(t,s)$ is a limit of a sequence of matrices of $SG$ FIOs of type I, 
with regular phase functions $\varphi_{jk}(t,s)$ belonging to $\Phr(c_h|t-s|)$, $c_h\ge1$, of class $C^1$ with respect to
$(t,s)\in\Delta_{T_0}$, and amplitudes
belonging to $C^1(\Delta_{T_0}, S^{0,0}(\R^{d}))$. 
\end{theorem}

Indeed, there are various techniques to switch from a Cauchy problem for an hyperbolic operator $L$ of order $M\ge 1$ to a 
Cauchy problem for a first order system \eqref{sys}, see, e.g., \cite{cordes, Coriasco:998.2, Morimoto:1}. In the approach we follow here,
which is the same used in \cite{Coriasco:998.3}, the key results for this aim
are the next Proposition \ref{prop:mlpf}, an adapted version of the so-called 
Mizohata Lemma of Perfect Factorization\footnote{See also \cite{kumano-go, Mizohata:1, Mizohata:2}, for the original version of such results.}, and the subsequent Lemma \ref{lem:Dt}. They are formulated for an hyperbolic operator $L$ with roots of constant multiplicities, and, of course, they hold true also in the more restrictive case of a strictly hyperbolic 
operator $L$, which coincides with the situation where $\displaystyle l=\max_{j=1,\dots,n}l_j=1\Leftrightarrow n=m$.

\begin{proposition}
\label{prop:mlpf}
Let $L$ be a hyperbolic operator with constant multiplicities
$l_{j}$, $j=1,\dots,n \le m$. Denote by $\theta_j\in G_j$, $j=1,\dots,n$, the distinct real roots of $\caL_m$ in \eqref{roots}.
Then, it is possible to factor $L$ as
\begin{equation}
\label{eq:Lfactor}
L = L_{n} \cdots L_{1} + \sum_{j=1}^m \Op(r_{j}(t)) D_{t}^{m-j}
\end{equation}
with
\begin{align}
\label{eq:Lj}
L_{j}&= (D_{t} - \Op(\theta_{j}(t)))^{l_{j}} + \sum_{k=1}^{l_{j}}
           \Op(h_{jk}(t)) \, (D_{t} - \Op(\theta_{j}(t)))^{l_{j}-k},
\\
\label{eq:hjk}
h_{jk} &\in C^\infty([0,T], S^{k-1, k-1}(\R^d)),
\quad
r_{j} \in C^\infty([0,T],S^{-\infty,-\infty}(\R^d)),
\quad
j=1, \dots, n, k = 1, \dots, l_{j}.
\end{align}
\end{proposition}
\noindent
The following corollary is an immediate consequence of Proposition \ref{prop:mlpf}, and is proved
by means of a reordering of the distinct roots $\theta_{j}$, $j=1,\dots,n$.
\begin{corollary}\label{cor:resort}
Let $\perm_{j}$, $j=1, \dots, n$, denote the reordering of the $n$-tuple $(1, \dots, n)$, given, for $k = 1, \dots, n$, by 
\begin{equation}\label{eq:resort}
\perm_{j}(k) =
  \begin{cases}
        j + k -1       & \mbox{for $j + k \le n+1$},
      \\ 
        j + k - n - 1 & \mbox{for $j + k  >  n+1$},
  \end{cases}
\end{equation}
That is, for $n\ge2$, $\perm_{1} = (1, \dots, n), \perm_{2} = (2, \dots, n,1),\dots, \perm_{n} = (n, 1, \dots, n-1)$. 
Then, under the same hypotheses of  Proposition \ref{prop:mlpf}, we have, for any $p=1, \dots, n$,
\begin{equation}
\label{eq:Lfactorbis}
L = L^{(p)}_{\perm_{p}(n)} \dots L^{(p)}_{\perm_{p}(1)} + 
    \sum_{j=1}^m \Op(r^{(p)}_{j}(t)) D_{t}^{m-j}
\end{equation}
with
\begin{equation}
\label{eq:Ljbis}
L^{(p)}_{j}= (D_{t} - \Op(\theta_{j}(t)))^{l_{j}} + \sum_{k=1}^{l_{j}}
           \Op(h^{(p)}_{jk}(t)) \, (D_{t} - \Op(\theta_{j}(t)))^{l_{j}-k},
\end{equation}
\begin{equation}
\label{eq:hjkbis}
h^{(p)}_{jk} \in C^\infty([0,T], S^{k-1, k-1}(\R^d)),
j=1, \dots, n, k = 1, \dots, l_{j},
\quad
r^{(p)}_{j} \in C^\infty([0,T],S^{-\infty,-\infty}(\R^d)),
j=1, \dots, m.
\end{equation}
\end{corollary}
\begin{remark}
	Of course, for $n=1$, we only have the single ``reordering'' $\varpi_1=(1)$, $l_1=l=m$, and 
	\[
		L = L^{(1)}_{1} +  \sum_{j=1}^m \Op(r^{(1)}_{j}(t)) D_{t}^{m-j}
	\]
with
\begin{align*}
L^{(1)}_{1}&= (D_{t} - \Op(\theta_{1}(t)))^m + \sum_{k=1}^{m}
           \Op(h^{(1)}_{1k}(t)) \, (D_{t} - \Op(\theta_{1}(t)))^{m-k},
\\
h^{(1)}_{1k} &\in C^\infty([0,T], S^{k-1, k-1}(\R^d)), 
k = 1, \dots, m,
\quad
r^{(1)}_{j} \in C^\infty([0,T],S^{-\infty,-\infty}(\R^d)),
j=1,\dots,m
\end{align*}
\end{remark}
With inductive procedures similar to those performed  in 
\cite{Cicognani-Zanghirati:997.1, Cicognani-Zanghirati:998.1} and 
\cite{Mizohata:2}, respectively, it is possible to prove the following Lemma \ref{lem:Dt}.
\begin{lemma}\label{lem:Dt}
Under the same hypotheses of Proposition \ref{prop:mlpf},
for all $k=0, \dots, m-1$, it is possible to
find symbols $\varsigma_{kpq} \in C^\infty([0,T],S^{k-q+l_{p}-n,k-q+l_p-n}(\R^d))$,
$p = 1, \dots, n$, $q = 0, \dots, l_{p}-1$, such that, for all $t\in[0,T]$,
\begin{equation*}
\theta^k = \sum_{p=1}^n 
                \left[\sum_{q=0}^{l_{p}-1} \varsigma_{kpq}(t) (\theta - \theta_{p}(t))^q\right]
                \cdot
                \left[\prod_{\afrac{1\le j \le n}{j\not=p}} (\theta - \theta_{j}(t))^{l_{j}}
                \right].
\end{equation*}
\end{lemma}
In the case of strict hyperbolicity, or, more generally, hyperbolicity with constant multiplicities, 
we can actually ``decouple'' the equations in \eqref{sys} into $n$ blocks of smaller dimensions, 
by means of the so-called \textit{perfect diagonalizer}, an element of
$C^\infty([0,T],\Op(S^{0,0}))$. Thus, the solution of \eqref{sys} can be reduced to the solution of $n$ independent smaller systems.
The principal part of the coefficient matrix of each one of such decoupled subsystems 
admits then a single distinct eigenvalue of maximum multiplicity,
so that it can be treated, essentially, like a
scalar $SG$ hyperbolic equations of first order. Explicitely, see, e.g., \cite{Coriasco:998.2,kumano-go}, 
\begin{theorem}\label{thm:perfd}
	Assume that the system \eqref{sys} is hyperbolic with constant multiplicities $\nu_j$, $j=1,\dots,N$, $\nu_1+\cdots+\nu_n=\nu$,
	with diagonal principal part $\kappa_1\in C^\infty([0,T],S^{1,1}(\R^d))$
	and $\kappa_0\in C^\infty([0,T], S^{0,0}(\R^d))$, both of them ($\nu\times \nu$)-dimensional matrices. 
	Then, there exist ($\nu\times\nu$)-dimensional matrices
	$\omega\in C^\infty([0,T],S^{0,0}(\R^d))$ and $\widetilde{\kappa}_0\in C^\infty([0,T],S^{0,0}(\R^d))$ such that
	\[
		\det(\omega)\asymp 1\Rightarrow \omega^{-1}\in	 C^\infty([0,T],S^{0,0}(\R^d)),
		\quad
		\widetilde{\kappa}_0=\mathrm{diag}(\widetilde{\kappa}_{01}, \dots, \widetilde{\kappa}_{0n}),
		\;
		\widetilde{\kappa}_{0j} \text{ ($\nu_j\times \nu_j$)-dimensional matrix},
	\]
	and
	\begin{equation}\label{eq:perfdiag}
		(D_t-\Op(\kappa_1(t))-\Op(\kappa_0(t)))\Op(\omega(t))
		-
		\Op(\omega(t))(D_t-\Op(\kappa_1(t))-\Op(\widetilde{\kappa}_0(t)))
		\in C^\infty([0,T],\Op(S^{-\infty,-\infty}(\R^d)).
	\end{equation}
\end{theorem}

%
\section{Function-valued solutions for semilinear SPDEs.}\label{sec:nonlin}
In this section we state and prove our main result of existence and uniqueness of a function-valued solution of the SPDE \eqref{eq:SPDE}, under suitable assumptions of hyperbolicity for the operator $L$, see \eqref{elle}, \eqref{roots}.
We work here with a class of operators with more general symbols than the
(polynomial) ones appearing in \eqref{elle}.
Namely, we consider operators of the form 
\beqs\label{elle2}
	L=D_t^m-\displaystyle\sum_{j=1}^m A_j(t,x,D_x) D_t^{m-j},
\eeqs
where $A_j(t)=\Op(a_j(t))$ are $SG$ pseudo-differential operators with symbols $a_{j}\in C^\infty([0,T], S^{j,j})$, $1\leq j\leq m$.
Notice that, of course, \eqref{elle} is a particular case of \eqref{elle2}.
The hyperbolicity condition on $L$ becomes
\begin{equation}\label{roots2}
	\caL_m(t,x,\tau,\xi)= \tau^m-\displaystyle\sum_{j=1}^m\tilde A_{j}(t,x,\xi)\tau^{m-j}=\prod_{j=1}^m\left(\tau-\tau_j(t,x,\xi)\right),
\end{equation}
where $\tilde A_j$ stands for the principal part of $A_j$, with characteristic roots 
$\tau_j(t,x,\xi)\in\R$, $\tau_j\in C^\infty([0,T];S^{1,1})$.
\\
Let us then consider the Cauchy problem 
\beqs\begin{cases}\label{cp}
 Lu(t,x) = \gamma(t,x, u(t,x)) + \sigma(t,x, u(t,x))\dot{\Xi}(t,x),\qquad (t,x)\in(0,T]\times\R^d
 \\
 D_t^j u(0,x)=u_j(x),\qquad x\in\R^d,\ 0\leq j\leq m-1,
\end{cases}
\eeqs 
where $L$ has the form \eqref{elle2}, under conditions \eqref{roots2} and either \eqref{def:strict} or \eqref{def:constmult} or \eqref{def:involutive}. We also assume that
$\gamma,\sigma:[0,+\infty)\times\R^d\times \R\longrightarrow \R$ are measurable functions, (at least locally-)Lipschitz-continuous, in our functional setting, with respect to the third variable, see Definition \ref{def:lip} and Theorem \ref{thm:semilinearweak} below for  the precise hypotheses. Here $\dot \Xi$ is the stochastic noise described in Subsection \ref{subnoise}.

We are interested in finding conditions on $L$, on the stochastic noise $\dot \Xi$, and on $\sigma, \gamma, u_j$, $j=0,\dots,m-1$,
such that \eqref{cp} admits a unique function-valued solution of the form \eqref{eq:mildsolutionSPDE}, following the stochastic integration theory presented in Subsection \ref{subH}.
To this aim, we need first to construct the distribution kernel $\Lambda$ that we are going to deal with. This is performed in Subsection \ref{subs:reduction}, through the following steps:
\begin{itemize}
	\item[-] reduction of the (formal) Cauchy problem 
	\beqs\begin{cases}\label{cplin}
 	Lu(t) = g(t)\qquad t\in(0,T]
 	\\
 	D_t^j u(0)=u_j,\qquad 0\leq j\leq m-1,
	\end{cases}
	\eeqs 
	where $L$ is the operator in \eqref{cp} and $g$ is a short notation for the right-hand side, to an equivalent first order system;
	\item[-] construction of the fundamental solution $E(t,s)$ for the system byTheorem \ref{thm:fundsolabs}, and then of its (formal) solution, following Section \ref{sec:fio};
	\item[-] construction of the distribution kernel $\Lambda$ and of the (formal) solution to \eqref{cplin} thanks to the equivalence of \eqref{cplin} and the corresponding first order system.   
\end{itemize}
For the sake of brevity, we will describe here the reduction procedure only in the case \ref{hyp:cm}, being \ref{hyp:st} a simpler subcase of \ref{hyp:cm} with $l=1$. In the case \ref{hyp:in}, we will shortly address some technical points that we are currently still investigating. Notice that all the results on $SG$-hyperbolic differential operators recalled in the previous Section \ref{sec:fio}, in particular, Proposition \ref{prop:mlpf} and Lemma \ref{lem:Dt}, still hold true for $SG$-hyperbolic operators of the form \eqref{elle2}. We adopt the same terminology and definitions also for this more general operators, with straightforward modifications, where needed.

Then, we need to understand the noise $\Xi$ in terms of a canonically associated Hilbert space $\mathcal H_\Xi$, so that we can define the stochastic integral with respect to a cylindrical Wiener 
process on $\mathcal H_\Xi$. This will be done in Subsection \ref{subs:pez}.

Finally, in Subsection \ref{subs:mainslin} we state and prove the main result of this paper, namely Theorem \ref{thm:semilinearweak}.
The conditions on the stochastic noise (explained in Subsection \ref{subs:pez}) will be given on the spectral measure 
$\mu$ corresponding to the correlation measure $\Gamma$ related to $\dot \Xi$. 
The operator $L$ is assumed to be hyperbolic (either strictly hyperbolic or weakly hyperbolic with constant multiplicity). The main conditions on $\sigma$ and $\gamma$ are Lipschitz continuity assumptions, which are typical in semilinear problems. 
We will also state in Theorem \ref{thm:semilinearinvolutive} a further result, corresponding to the involutive case \ref{hyp:in}, whose detailed proof will be given elsewhere.

\begin{remark} With respect to the existing literature, in particular \cite{peszat}, we allow here for general hyperbolic equations of higher orders, coefficients depending both on time and space, and possibly with a polynomial growth with respect to $x$. We remark that in the strictly hyperbolic case we obtain the compatibility condition \eqref{eq:meascm} with $l=1$ (set $l=1$ in Theorem \ref{thm:semilinearweak}), which exactly corresponds, for $m=2$, to the one obtained in \cite{peszat}.
\end{remark}

The proof of the main Theorem \ref{thm:semilinearweak} starts with the construction of the (formal) solution $u$ to \eqref{cp}, and then deals with existence and uniqueness of $u$ under the assumption \ref{hyp:cm} through a fixed point theorem, in view of the fundamental Lemma \ref{lem:weightedpesz} below, which is the key of the proof of the Theorem and the main result of Subsection \ref{subs:pez}.

%
\subsection{Reduction to a first order system and construction of the distribution kernel $\Lambda$.}\label{subs:reduction}
Let us denote by $\theta_j$, $j=1,\dots, n$, the distinct values of the roots $\tau_k$, $k=1, \dots, m$, and
with $\varpi_p$, $p=1,\dots,n$, the reorderings of the $n$-tuple $(1,\dots,n)$ defined in \eqref{eq:resort}.

The equivalence of the Cauchy problems for the equation $Lu(t)=g(t)$ and a $1\times 1$ system \eqref{sys} is then
trivial for $m=1$. For $m\ge 2$, we will now define a ($nm$)-dimensional vector of unknown $W$ and construct a corresponding linear first order hyperbolic system, with diagonal principal part
and constant multiplicities, equivalent to $Lu(t)=g(t)$. 

Let us set, for convenience, with 
the notation introduced in Corollary \ref{cor:resort},
\[
l^{(p,k)} =
   \begin{cases}
     0,                                     & k = 0,
    \\
    \displaystyle \sum_{1 \le j \le k} l_{\varpi_{p}(j)}, & 1 \le k \le n-1, \text{ if $n\ge2$},
    \\
    m,					& k=n,
   \end{cases}
\quad
L^{(p,k)} =
   \begin{cases}
     I,                                                & k = 0,
    \\
     L^{(p)}_{\varpi_{p}(k)} \cdots L^{(p)}_{\varpi_{p}(1)},   & 1 \le k \le n-1, \text{ if $n\ge2$},
   \end{cases}
\]
$p=1,\dots,n$, and define
\begin{equation}\label{eq:Wpbis}
W^{(p)}_{l^{(p,k)}+j+1}(t) = (D_t-\Op(\theta_{\varpi_{p}(k+1)}(t)))^j L^{(p,k)} u(t),
\quad
p= 1, \dots, n, \, k=0,\dots,n-1,\, j = 0, \dots, l_{\varpi_{p}(k+1)} - 1.
\end{equation}
Using Lemma \ref{lem:Dt}, we can express the $t$ derivatives of $u$ in terms of the components of $W$
from \eqref{eq:Wpbis}. In fact:

\begin{lemma}\label{lem:Dtkbis}
Under the hypotheses of Lemma \ref{lem:Dt}, for all $k=1,\dots,m-1$, $p=1,\dots,n$, it is possible to find symbols 
$w^{(p)}_{kj}\in C^\infty([0,T],S^{j,j}(\R^d))$,
$j=1,\dots,k$, such that, with the ($nm$)-dimensional vector $W$ defined in \eqref{eq:Wpbis},
\begin{equation}\label{eq:Dtkbis}
	D_t^ku(t)=\sum_{j=1}^k\Op(w^{(p)}_{kj}(t))W^{(p)}_{k-j+1}(t)+W^{(p)}_{k+1}(t).
\end{equation}
\end{lemma}
\noindent
By the definition \eqref{eq:Wpbis}, we find the extension of \eqref{eq:Dtkbis} to $k=0$ in the form
$u(t)=W^{(p)}_1(t)$, $p=1,\dots,n$. 
Using Lemma \ref{lem:Dtkbis} we see that \eqref{eq:Lfactorbis}, \eqref{eq:Wpbis} and \eqref{eq:Dtkbis} give rise to a block diagonal linear system in the $nm$ unknown $W^{(p)}_{l^{(p,k)}+j+1}(t)$ with blocks labeled by $p=1,\dots,n$, of the type
\begin{equation}\label{eq:Wmbis}
	\left\{
	\begin{aligned}
	&\dots,
	\\
	(D_t-\Op(\theta_{\varpi_p(1)}(t)))W^{(p)}_{j+1}(t) &= \phantom{-}W^{(p)}_{j+2}(t), 
	\quad j=0,\dots,l_{\varpi_p(1)}-2, \text{ if $l_{\varpi_p(1)}\ge2$}, 
	\\
	(D_t-\Op(\theta_{\varpi_p(1)}(t)))W^{(p)}_{l^{(p,1)}}(t) & = 
	-\sum_{k=1}^{l_{\varpi_p(1)}}\Op(h^{(p)}_{\varpi_p(1)k}(t))W^{(p)}_{l^{(p,1)}-k+1}(t)+
	W^{(p)}_{l^{(p,1)}+1}(t),
	\\
	(D_t-\Op(\theta_{\varpi_p(2)}(t)))W^{(p)}_{l^{(p,1)}+j+1}(t) &= \phantom{-}W^{(p)}_{l^{(p,1)}+j+2}(t), 
	\quad j=0,\dots, l_{\varpi_p(2)}-2, \text{ if $l_{\varpi_p(2)}\ge2$, $n\ge2$}, 
	\\
	(D_t-\Op(\theta_{\varpi_p(2)}(t)))W^{(p)}_{l^{(p,2)}}(t) & = 
	-\sum_{k=1}^{l_{\varpi_p(2)}}\Op(h^{(p)}_{\varpi_p(2)k}(t))W^{(p)}_{l^{(p,2)}-k+1}(t)+
	W^{(p)}_{l^{(p,2)}+1}(t), \text{ if $n\ge2$},
	\\
	&\dots,
	\\
	\displaystyle (D_t-\Op(\tau_{\varpi_p(n)}(t)))W^{(p)}_{m}(t)&=
	-\sum_{k=1}^{l_{\varpi_p(n)}}\Op(h^{(p)}_{\varpi_p(n)k}(t))W^{(p)}_{m-k+1}(t)
	\\
	-\sum_{j=1}^{m-1}&\left(\sum_{q=1}^{m-j}\Op(r^{(p)}_j(t))\circ\Op(w^{(p)}_{m-j,q}(t))W^{(p)}_{m-j-q+1}(t)+
	\Op(r^{(p)}_j(t))W^{(p)}_{m-j+1}(t)\right)
	\\
	&\hspace{-10mm}-\Op(r^{(p)}_m(t))W^{(p)}_1(t)+g(t),
	\\
	&\dots
	\end{aligned}
	\right.
\end{equation}
and equivalent, block by block, to the equation $Lu(t)=g(t)$.

As it is very well-known in the usual hyperbolic theory, in the case of weak hyperbolicity the principal term does not provide enough information, by itself, to imply well-posedness of the Cauchy problem. In other words, lower order terms are also relevant in this case, and one needs to impose additional conditions on them. We will then assume that $L$ satisfies the $SG$-Levi 
condition
\begin{equation}
\label{eq:LC}
h^{(p)}_{jk} \in C^\infty([0,T], S^{0,0}(\R^d)), \quad
p,j=1, \dots, n, k=1, \dots, l_{j},
\end{equation}
see Corollary \ref{cor:resort}.

\begin{remark}
	Let us observe that, indeed, \eqref{eq:LC} needs to be fulfilled only for a single value of $p=1,\dots,n$. 
	Also, \eqref{eq:LC} is automatically fulfilled when $L$ is strictly $SG$-hyperbolic.
	If $L$ satisfies \eqref{eq:LC} we will also say that $L$ is of Levi type. 
\end{remark}

It is clear, in view of the 
calculus of $SG$ pseudodifferential operators, the fact that $r^{(p)}_j\in C^\infty([0,T],S^{-\infty,-\infty})$, $p=1,\dots,n$,
and the inclusions among the $SG$ 
symbols, that the system \eqref{eq:Wmbis} is a hyperbolic first order linear system of the form \eqref{sys}, where:

\noindent
- the ($nm\times nm$)-dimensional, block-diagonal matrix $\kappa_1\in C^\infty([0,T], S^{1,1})$ is given by 
	$\kappa_1=\mathrm{diag}(\kappa_{11},\dots,\kappa_{1n})$, with each block defined by
	$$\kappa_{1p}=\mathrm{diag}(\underbrace{\theta_{\omega_p(1)},\dots,\theta_{\omega_p(1)}}_{\text{$l_{\omega_{p(1)}}$ times}},\underbrace{\theta_{\omega_p(2)},\dots,\theta_{\omega_p(2)}}_{\text{$l_{\omega_{p(2)}}$ times}}, \ldots, \underbrace{\theta_{\omega_p(n)},\dots,\theta_{\omega_p(n)}}_{\text{$l_{\omega_{p(n)}}$ times}} ),\ p=1,\dots,n;$$
\\
- 	the ($nm\times nm$)-dimensional, block-diagonal matrix
	$\kappa_0\in C^\infty([0,T], S^{0,0})$ is given by $\kappa_0=\mathrm{diag}(\kappa_{01},\dots,\kappa_{0m})$
	%
	with suitable matrices $\kappa_{0p}$ having entries in $C^\infty([0,T], S^{0,0})$, $p=1,\dots,n$;
\\
- the right-hand side is
	\[
	Y(t)=(\underbrace{G(t),\dots,G(t)}_{\text{$n$ times}})^t, 
	\quad
	G(t)=(\underbrace{0,\dots,0}_{\text{$m-1$ times}},g(t))^t.
\]
The initial data $W_0$ is obtained by $W_0=\Op(b)U_0$, with $U_0=(u_0, \dots, u_{m-1} )^t$ and a ($mn\times m$)-dimensional 
block-matrix symbol $b$ with the following structure:
\begin{equation}\label{eq:bm}
	b=\begin{pmatrix}
		b^{(1)}\rule{0mm}{6mm}
		\\
		\underline{\hspace{6mm}}
		\\
		\dots
		\\
		\underline{\hspace{6mm}}\rule{0mm}{3mm}
		\\
		b^{(n)}\rule{0mm}{6mm}
		\\
		\hspace{6mm}
	\end{pmatrix},
	\quad
	b^{(p)}=\begin{pmatrix}
		1 & 0 & 0 & 0 & \dots
		\\
		b^{(p)}_{10} & 1 & 0 & 0 &\dots
		\\
		b^{(p)}_{20} & b^{(p)}_{21} & 1 & 0 & \dots
		\\
		\dots & \dots & \dots & \dots &\dots
	\end{pmatrix}, p=1,\dots,n,
\end{equation}
and the ($m\times m$)-dimensional matrices $b^{(p)}$ satisfying 
\begin{itemize}
	\item[-] if $m\ge2$, $b^{(p)}_{jk}\in S^{j-k,j-k}$, $j > k$, $j=1,\dots, m-1$, $k=0,\dots,j-1$,
	\item[-] $b^{(p)}_{jj}=1\in S^{0,0}$, $j=0,\dots,m-1$,
	\item[-] if $m\ge2$, $b^{(p)}_{jk}=0$, $j<k$, $j=0,\dots,m-2$, $k=j+1,\dots, m-1$,
\end{itemize}
$p=1,\dots,m$.

\begin{remark}
	Consider, for instance, the case $n=1$, that is, $\caL_m$ admits a unique real root $\theta_1=\tau_1$
	of maximum multiplicity $l=l_1=m$.
	Then, there is a single ``reordering'' $\varpi_1=(1)$,
	the vector $W$ has $m$ components, $W=(W^{(1)}_1, \dots, W^{(1)}_m)$, 
	and \eqref{eq:Wmbis} consists of a single block
	of $m$ equations. Namely, in view of Corollary \ref{cor:resort}, assuming $n\ge2$ and
	dropping everywhere the ${^{(1)}}$ label, \eqref{eq:Wpbis} reads, in this case,
	\begin{align*}
		W_1(t)&=u(t),
		\\
		W_2(t)&=(D_t-\Op(\tau_1(t)))u(t)=(D_t-\Op(\tau_1(t)))W_1(t),
		\\
		&\dots,
		\\
		W_{m}(t)&=(D_t-\Op(\tau_1(t)))^{m-1}u(t)=(D_t-\Op(\tau_1(t)))W_{m-1}(t),
	\end{align*}
	while $Lu(t)=g(t)$ is then equivalent to 
	\begin{align*}
		(D_t-\Op(\tau_1(t)))^mu(t)&+
		\sum_{k=1}^{m}\Op(h_{1k}(t))(D_t-\Op(\tau_1(t)))^{m-k}u(t)+\sum_{j=1}^m \Op(r_{j}(t)) D_{t}^{m-j}u(t)=g(t)
		\\
		&\Leftrightarrow
		\\
		(D_t-\Op(\tau_1(t)))W_m(t)&=-\sum_{k=1}^{m}\Op(h_{1k}(t))W_{m-k+1}(t)
		\\
		-\sum_{j=1}^{m-1}&\left(\sum_{q=1}^{m-j}\Op(r_j(t))\circ\Op(w_{m-j,q}(t))W_{m-j-q+1}(t)+
		\Op(r_j(t))W_{m-j+1}(t)\right)
		\\
		&\hspace{-10mm}-\Op(r_m(t))W_1(t)+g(t),
	\end{align*}
	that is,
	\[
		\left\{
		\begin{aligned}
			(D_t-\Op(\tau_1(t)))W_1(t)&=\phantom{-}W_2(t)
			\\
			\dots
			\\
			(D_t-\Op(\tau_1(t)))W_{m-1}(t)&=\phantom{-}W_m(t)
			\\
		(D_t-\Op(\tau_1(t)))W_m(t)&=-\sum_{k=1}^{m}\Op(h_{1k}(t))W_{m-k+1}(t)
		\\
		-\sum_{j=1}^{m-1}&\left(\sum_{q=1}^{m-j}\Op(r_j(t))\circ\Op(w_{m-j,q}(t))W_{m-j-q+1}(t)+
		\Op(r_j(t))W_{m-j+1}(t)\right)
		\\
		&\hspace{-10mm}-\Op(r_m(t))W_1(t)+g(t),			
		\end{aligned}
		\right.
	\]
	which has the form \eqref{sys} with $Y(t)=(\underbrace{0,\dots,0}_{m-1\text{ times}},g(t))^t$, as claimed,
	since $\kappa_1(t)=\mathrm{diag}(\tau_1(t),\dots,\tau_1(t))$, while
	the coefficients of the components of $W$ in the right-hand sides of the equations
	are all symbols of order $(0,0)$, since $S^{-\infty,-\infty}\subset S^{0,0}$.
\end{remark}
In this situation, 
by an extension of the results in \cite{Coriasco:998.2, Coriasco:998.3}, we can give an explicit form 
to the fundamental solution $E(t,s)$ in Theorem \ref{thm:fundsol}, in terms of (smooth families of) $SG$ FIOs of type I, modulo smoothing
remainders.
With the results of Theorem \ref{thm:perfd} at hand, we 
solve, by means of the so-called
\textit{geometrical optics} (or FIOs) method, the system
\begin{equation}\label{eq:sysz}
\begin{cases}
	(D_t - \Op(\kappa_1(t)) - \Op(\widetilde{\kappa}_0(t)))\widetilde{E}(t,s) = 0, & t\in[0,T_0],
	\\
	\widetilde{E}(s,s)=I, & s\in[0,T_0).
\end{cases}
\end{equation}
Notice that the \textit{approximate solution operator} $\widetilde{A}(t,s)$, $(t,s)\in\Delta_{T_0}$, 
in terms of $SG$ FIOs solves the corresponding
operator problem up to smoothing remainders. Namely, the FIOs family $\widetilde{A}(t,s)$ solves the system
\begin{equation}\label{eq:syszopr}
\begin{cases}
	(D_t - \Op(\kappa_1(t)) - \Op(\widetilde{\kappa}_0(t)))\widetilde{A}(t,s) = \widetilde{R}_1(t,s), & (t,s)\in \Delta_{T_0}, 
	\\
	\widetilde{A}(s,s)=I+\widetilde{R}_2(s), & s\in[0,T_0),
\end{cases}
\end{equation}
where $\widetilde{R}_1$ and $\widetilde{R}_2$ are suitable smooth families of operators in
 $\caO(-\infty,-\infty)$, coming from the solution method, see
\cite{cordes, coriasco99, Coriasco:998.2, Coriasco:998.3, kumano-go} for more details. It turns out that 
$\widetilde{A}(t,s)$ belongs to $\caO(0,0)$ for any 
$(t,s)\in\Delta_{T_0}$. Explicitely, 
\begin{align*}
\widetilde{A}(t,s)&=\mathrm{diag}(\widetilde{A}^{(1)}(t,s),\dots,\widetilde{A}^{(m)}(t,s)),
\\
\widetilde{A}^{(p)}(t,s)&=\mathrm{diag}(\Op_{\varphi_{\varpi_p(1)}(t,s)}(a^{(p)}_1(t,s)), \dots,\Op_{\varphi_{\varpi_p(m)}(t,s)}(a^{(p)}_m(t,s))), 
p=1,\dots,m,
\end{align*}
with phase functions
$\varphi_j\in C^\infty(\Delta_{T_0}, \Phr(\lambda))$, $\lambda=\lambda(T_0)$ suitably small, solutions of the eikonal equations \eqref{eik}
with $\tau_j$ in place of $\varkappa$, and symbols $a^{(p)}_j\in C^\infty(\Delta_{T_0}, S^{0,0})$, $p,j=1,\dots, m$, see \cite{Coriasco:998.2}. In fact,
in this case, the system can be diagonalized block by block.

Solving the equations in \eqref{eq:sysz} modulo smoothing terms
is enough for our aims, as we will see below. Indeed, we have the following result (see \cite{ACSlinear} for its proof).
\begin{proposition}\label{prop:diffsm}
	Under the hypotheses \eqref{elle2}, \eqref{roots2}, let
	$A(t,s)=\Op(\omega(t))\circ\widetilde{A}(t,s)\circ\Op(\omega_{-1})(s)$, with
	$\widetilde{A}(t,s)$ solution
	of \eqref{eq:syszopr}, $(t,s)\in\Delta_{T_0}$, and 
	$\Op(\omega_{-1})(s)$ parametrix of the perfect diagonalizer $\Op(\omega(s))$, $s\in[0,T]$.
	Then, the solution ${E}(t,s)$ of \eqref{tocheck} and the operator family $A(t,s)$
	satisfy ${E}-{A}\in C^\infty(\Delta_{T_0},\Op(S^{-\infty,-\infty}(\R^d)))$.
\end{proposition}
\begin{remark}\label{rem:ker}
	Proposition \ref{prop:diffsm} means that the Schwartz kernels of ${E}$ and ${A}$ differ by a family
	of elements of $\caS(\R^{2d})$, smoothly depending on $(t,s)\in\Delta_{T_0}$.
\end{remark}
Using Proposition \ref{prop:diffsm}, by repeated applications of Theorem \ref{thm:compi}, we finally
obtain 
\begin{equation}\label{eq:e}
	E(t,s)=E_0(t,s)+R(t,s),\quad (t,s)\in\Delta_{T_0},
\end{equation}
where 
\begin{itemize}
	\item[-] $E_0$ is a ($nm\times nm$)-dimensional matrix of operators in $\caO(0,0)$ given by
	\[
		E_0(t,s)=\left(\sum_{p=1}^n\Op_{\varphi_p(t,s)}(e_{pjk}(t,s))\right)_{j,k=0,\dots,nm-1},
	\]
	with the regular phase-functions $\varphi_p(t,s)$, solutions of the eikonal equations associated with $\tau_p$,
	and symbols $e_{pjk}(t,s)\in S^{0,0}$, $j,k=0,\dots,nm-1$, $p=1,\dots,n$, smoothly depending on $(t,s)\in\Delta_{T_0}$;
	\item[-] $R$ is a ($nm\times nm$)-dimensional matrix of elements in $C^\infty(\Delta_{T_0},\Op(S^{-\infty, -\infty}))$,
	operators with kernel in $S(\R^{2d})$, smoothly depending on $(t,s)\in\Delta_{T_0}$, that is,
	\[
		R=(\Op(r_{jk}(t,s)))_{j,k=0,\dots,nm-1},
	\]
	with symbols $r_{jk}\in C^\infty(\Delta_{T_0},S^{-\infty,-\infty})$, $j,k=0,\dots,nm-1$,
	 collecting the remainders of the compositions in 
	$\Op(\omega)\circ \widetilde{A}\circ\Op(\omega_{-1})$ and the difference $E-A$.
\end{itemize}  
The next Lemma \ref{lem:solregbis} 
from \cite{Coriasco:998.3}, see also \cite{Cicognani-Zanghirati:997.1, Cicognani-Zanghirati:998.1} and
\cite{Mizohata:2}, is the key result to achieve, from \eqref{eq:e} and the expressions of $E_0$ and $R$,
the correct regularity of $u$.
\begin{lemma}\label{lem:solregbis}
There exists a ($m\times mn$)-dimensional matrix $\varUpsilon_n\in C^\infty([0,T_0],S^{0,0}(\R^d)$ such that
the $k$-th row consists of symbols of order $(l-m+k,l-m+k)$, $k=0,\dots,m-1$, and
\[
	\begin{pmatrix}
	u(t)
	\\
	\dots
	\\
	D_{t}^{m-1}u(t)
	\end{pmatrix} = \Op(\varUpsilon_n(t))W(t), \quad t\in[0,T_0].
\]
\end{lemma}
Now assume for a moment that $g\in C([0,T], H^{z,\zeta})$, $(z,\zeta)\in\R^2$. Then, the Cauchy problem for the first order system \eqref{sys} with $s=0$, 
equivalent to \eqref{cplin}, fulfills all the assumptions of Theorem \ref{thm:fundsolabs}.
An application of Theorem \ref{thm:fundsolabs}, together with \eqref{eq:e} and Lemma \ref{lem:solregbis} initially gives
 \[
	\begin{pmatrix}
	u(t)
	\\
	\dots
	\\
	D_{t}^{m-1}u(t)
	\end{pmatrix}
	=[\Op(\varUpsilon_n(t))\circ(E_0(t,0)+R(t,0))\circ\Op(b)]U_0+i\ds\int_0^t[\Op(\varUpsilon_n(t))\circ(E_0(t,s)+R(t,s))]Y(s)ds, t\in[0,T_0].
\]
Then, taking into account that the only non-vanishing entries of $Y$ coincide with $g$, computations with matrices, 
the structure of the entries of $\varUpsilon_n$ and $b$, and further applications of Theorem \ref{thm:compi} give
\beqs\label{eq:uexplicitbis}
\begin{aligned}
u(t)&=\sum_{j=0}^{m-1}\left[\sum_{p=1}^n\Op_{\varphi_p(t,0)}(c^0_{pj}(t))+\Op({r}^0_j(t))\right]u_j
+i \int_0^t\left[\sum_{p=1}^n\Op_{\varphi_p(t,s)}(c^1_{p}(t,s))+\Op({r}^1(t,s))\right]g(s)ds,
\\
&=v_0(t)+\int_0^t\int_{\R^d}\Lambda(t,s,.,y)g(s,y)\,dyds,
\end{aligned}
\eeqs
where
\begin{itemize}
	\item[-] the phase functions $\varphi_p$ are solution to the eikonal equations \eqref{eik}, with $\theta_p$ in place of $\varkappa$,
	$p=1,\dots,n$;
	\item[-] $c^0_{pj}\in 
	C^\infty([0,T_0],S^{l-1-j,l-1-j})$, $p=1,\dots,n$, ${r}^0_j\in C^\infty([0,T_0],S^{-\infty,-\infty})$,
	$j=0,\dots,m-1$, so that $\displaystyle v_0\in \bigcap_{j\ge0}C^j([0,T_0],$ $H^{z+m-l-j,\zeta+m-l-j})$;
	\item[-] $\Lambda\in C^\infty(\Delta_{T_0},\caS^\prime)$ is, for any $(t,s)\in\Delta_{T_0}$,
	the Schwartz kernel of the operator
	\begin{equation}\label{eq:truekersolbis}
		{E}_{l-m}(t,s)=i\left[\sum_{p=1}^n\Op_{\varphi_p(t,s)}(c^1_p(t,s))+\Op(r^1(t,s))\right],
	\end{equation}
	with $c^1_{p}\in C^\infty(\Delta_{T_0},S^{l-m,l-m})$, $p=1,\dots,m$, ${r}^1\in C^\infty(\Delta_{T_0},S^{-\infty,-\infty})$,
	so that also 
	\[
		\int_0^t\int_{\R^d}\Lambda(t,s,.,y)g(s,y)\,dyds\in 
		\bigcap_{j\ge0}C^j([0,T_0], H^{z+m-l-j,\zeta+m-l-j}).
	\]
\end{itemize}
Notice the usual abuse of notation, using the kernel $\Lambda(t,s)$ in the \textit{distributional integral} 
in \eqref{eq:uexplicitbis}.
By Proposition \ref{thm:smoothing}, $\Lambda(t,s)$ differs by an element of $C^\infty(\Delta_{T_0},\caS)$
from the kernel of 
\begin{equation}\label{eq:kersolappbis}
	\widetilde{E}_{l-m}(t,s)=i\sum_{p=1}^n\Op_{\varphi_p(t,s)}(c^1_{p}(t,s)).
\end{equation}

\subsection{Admissible spectral measures for Hilbert space valued stochastic integrals.}\label{subs:pez}
In this subsection we want to make sense of the stochastic integral appearing in \eqref{eq:mildsolutionSPDE} as a stochastic integral with respect to a cylindrical Wiener process on a Hilbert space, as described in Subsection \ref{subH}. We know from Section \ref{subs:reduction} that, in the stochastic integral appearing in \eqref{eq:mildsolutionSPDE}, $\Lambda$ is the kernel of (a linear combination of) FIOs $E_{l-m}$, with amplitudes of order $(l-m,l-m)$, where $l$ stands for the maximum multiplicity of the characteristic roots ($l=1$ in the case of a strictly hyperbolic operator, $1<l\le m$ in the constant multiplicities case).
To give meaning to
\beqs\label{lintegrale}
\int_0^t\int_\Rd \Lambda(t,s,x,y)\sigma(s,y, u(s,y))\dot\Xi(s,y)dyds=\int_0^tE_{l-m}(t,s)\sigma(s,u(s))d\Xi(s),
\eeqs
we first introduce the so-called Cameron-Martin space associated with $\Xi$.
Given the Gaussian process $\Xi$ described in Section \ref{subnoise}, let us define 
\begin{equation}\label{cammart}
\mathcal H_\Xi=\{\widehat{\varphi\mu}\colon\varphi\in L^2_{\mu,s}(\R^d)\},
\end{equation}
where $\mu$ is the spectral measure associated with the noise $\Xi$, and $L^2_{\mu,s}$ is the space of symmetric functions in $L^2_\mu$, i.e. 
$\check{\varphi}(x)=\varphi(-x)=\varphi(x)$, $x\in\R^d$, and $\int_{\R^d}|\varphi(x)|^2\,\mu(dx)<\infty$. 
Clearly, $\mathcal H_\Xi\subset\mathcal S'(\R^d)$. The space $\mathcal H_\Xi$, endowed with the inner product
\[\langle\widehat{\varphi\mu},\widehat{\psi\mu}\rangle_{\mathcal H_\Xi}:=\langle\varphi,\psi\rangle_{L^2_\mu}, \quad \forall\varphi,\psi\in L^2_{\mu,s}(\R^d)\]
with corresponding norm \[||\widehat{\varphi\mu}||_{\mathcal H_\Xi}^2=||\varphi||_{L^2_\mu}^2\]
turns out to be a real separable Hilbert space, and it is the so-called
"Cameron-Martin space" of $\Xi$, see \cite[Propostition 2.1]{peszat}. Thus, $\Xi$ is a  cylindrical  Wiener process on $(\mathcal H_\Xi,\langle\cdot,\cdot\rangle_{\mathcal H_\Xi})$ which takes values in any Hilbert space $\mathcal U$ such that the embedding $\mathcal H_\Xi\hookrightarrow \mathcal U$ is an Hilbert-Schmidt map.

The following Lemma \ref{lem:weightedpesz} shows that the multiplication operator $\mathcal H_\Xi\ni\psi\mapsto E_{l-m}(t,s)\sigma(s,u)\cdot \psi$ is  Hilbert-Schmidt from $\mathcal H_\Xi$ to $H^{z+m-l,\zeta}$, under suitable assumptions on $\sigma$. Therefore, \eqref{lintegrale} is well-defined as stochastic integral with respect to a cylindrical  Wiener process on $(\mathcal H_\Xi,\langle\cdot,\cdot\rangle_{\mathcal H_\Xi})$ which takes values in $H^{z+m-l,\zeta}$.

\begin{definition}\label{def:lip}
The class $\mathrm{Lip}(z,\zeta,r,\rho)$, for given $z,\zeta,r,\rho\in\R$, $r,\rho\ge0$,
consists of all measurable functions $g:[0,T]\times\R^d\times\R\longrightarrow\C$ such that
there exists a real-valued, non negative,
$C_t=C(t)\in C[0,T]$, fulfilling the following:
\begin{itemize}
\item for every $w\in H^{z+r,\zeta+\rho}(\R^d)$, $t\in[0,T]$, we have
$\|g(t,\cdot,w)\|_{z,\zeta}\leq C(t)(1+\|w\|_{z+r,\zeta+\rho})$;
\item for every $w,v\in H^{z+r,\zeta+\rho}(R^d)$, $t\in[0,T]$, we have  
$\|g(t,\cdot,w)-g(t,\cdot,v)\|_{z,\zeta}\leq C(t)\|w-v\|_{z+r,\zeta+\rho}$.
\end{itemize}
\end{definition}

\begin{remark}\label{rem:lipbase}
		In Definition \ref{def:lip} we can actually relax the hypotheses, 
		and ask that the stated properties hold for $w,v\in U$,
		with $U$ a suitable open subset of $H^{w,\omega}(\R^d)$, for some 
		$w\ge z+r$, $\omega\ge \zeta+\rho$ (typically, a sufficiently small neighbourhood of
		the initial data of the Cauchy problem). In this case, we indicate
		the corresponding set by $\mathrm{Lip_{loc}}(z,\zeta,r,\rho)$.
\end{remark}

\begin{remark}\label{rem:lip}
Let $g:[0,T]\times\R^d\times\R\longrightarrow\R$ be measurable and
$\zeta=\rho=0$. Assume that
there exists a real-valued, non negative,
$C_t=C(t)\in C[0,T]$, satisfying
\begin{itemize}
\item for every $w\in\R$, $x\in\R^d$, $t\in[0,T]$, we have
$|g(t,x,w)|\leq C(t)(|\kappa(x)|+|w|)$, for some $\kappa \in H^{z,\zeta}(\R^d)$, and
\item for every $w,v\in\R$, $x\in\R^d$, $t\in[0,T]$, we have  $|g(t,x,w)-g(t,x,v)|\leq C(t)|w-v|$.
\end{itemize}
Then, $g\in\mathrm{Lip}(z,0,r,0)$. In fact, for some $C>0$,
	\begin{align*}
		\| g(t,\cdot,w) \|_{z,0}^2&=
		\| \jap^z g(t,\cdot,w) \|_{L^2}^2\leq C_t^2\|
		\jap^z (|\kappa|+|w|) \|_{L^2}^2\le 2C_t^2(\|\kappa\|^2_{z,0}+\|w\|^2_{z,0})
		\le C^2 C_t^2(1+\|w\|_{z+r,0})^2,
	\end{align*}
	and similarly for the Lipschitz continuity with respect to the third variable, cfr. \cite{peszat}.
\end{remark}

\begin{remark}\label{rem:intpowers}
	Let $g(t,x,w)=w^n$, $n\in\N$. Then 
	$g\in\mathrm{Lip_{loc}}(z,\zeta,r,\rho)$,
	when $z,r,\rho\ge0$,
	$\zeta>\frac{d}{2}$. In fact, 
	when $w\in H^{z+r,\zeta+\rho}(\R^d)$ is such that $\|w\|_{z+r,\zeta+\rho}\le R$,
	$$
	\| w^n \|_{z,\zeta}\le C \| w^n \|_{nz,\zeta}\le C \|w \|_{z,\zeta}^n\le 
	\widetilde{C} R^{n-1}\|w\|_{z+r,\zeta+\rho},
	$$
	for the algebra properties of the Sobolev-Kato spaces, see e.g.\ \cite[Proposition 2.2]{AC06}.
\end{remark}

\begin{lemma}\label{lem:weightedpesz}
Let $E_{l-m}(t,s)$ be a family of FIOs with amplitudes of order $(l-m,l-m)$, $0\leq l\leq m$, parametrized by
$0\le s\le t \le T$, and 
$\sigma\in \mathrm Lip(z,\zeta,m-l,0)$. If the spectral measure satisfies
\beqs
\sup_{\eta\in\R^d}\ds\int_{\R^d}\frac{1}{(1+|\xi+\eta|^2)^{m-l}}\mu(d\xi)<\infty,
\eeqs
(cfr \eqref{eq:meascm}),
then, for every $w\in H^{z+m-l,\zeta}(\R^d)$, the operator $\Phi(t,s)=\Phi_{l,m,\sigma,w}(t,s)\colon
\psi\mapsto E_{l-m}(t,s)\sigma(s,w) \psi$ belongs to $L_0^2(\mathcal H_\Xi, H^{z+m-l,\zeta}(\R^d))$. 
Moreover, the Hilbert-Schmidt norm of $\Phi(t,s)$ can be estimated by 
\[
\|\Phi(t,s)\|_{L_0^2(\mathcal H_\Xi, H^{z+m-l,\zeta})}^2\leq 
C^2_{t,s}(1+\|w\|_{z+m-l,\zeta})^2
\sup_{\eta\in\Rd}\int_\Rd \frac{1}{(1+|\xi+\eta|^2)^{m-l}}\mu(d\xi),
\]
for some $C_{t,s}>0$.
\end{lemma}

\begin{remark} Lemma \ref{lem:weightedpesz} is the key result to prove Theorems \ref{thm:semilinearweak} and \ref{thm:semilinearinvolutive}. 
It is a generalization, for higher order equations
and different functional spaces, of Lemma 2.2 in \cite{peszat}. There, the author deals with the case
$m=2$ and $l=1$, related to the wave equation, and works with a multiplication
operator
by a test function $w$, obtaining an estimate of the corresponding Hilbert-Schmidt norm
involving a weighted $L^2$ norm of $w$. 
\end{remark}

\begin{proof}[Proof of Lemma \ref{lem:weightedpesz}]
Let us fix an orthonormal basis $\{e_k\}_{k\in\N}=\{\widehat{f_k\mu}\}_{k\in\N}$ of $\mathcal H_\Xi$, where $\{f_k\}_{k\in\N}$ is an orthonormal basis in $L^2_{\mu,s}$. We compute
\beqs\nonumber
||\Phi(t,s)||_{L_2^0(\mathcal H_\Xi, H^{z+m-l,\zeta})}^2&=&\sum_{k\in\N}||E_{l-m}(t,s)\sigma(s,w)\widehat{f_k\mu}||_{H^{z+m-l,\zeta}}^2
\\\nonumber
&=&\sum_{k\in\N}||\langle D\rangle^{l-m}\langle D\rangle^{m-l}\langle \cdot\rangle^{z+m-l}\langle D\rangle^\zeta E_{l-m}(t,s)\sigma(s,w)\widehat{f_k\mu}||_{L^2}^2
\\\nonumber
&=&\sum_{k\in\N}||\langle D\rangle^{l-m}\widetilde E(t,s)\sigma(s,w)\widehat{f_k\mu}||_{L^2}^2
\\\label{las}
&=&(2\pi)^{-d}\sum_{k\in\N}\int_{\R^d}\langle\xi\rangle^{2(l-m)}\left\vert\mathcal F\left(
\widetilde E(t,s)\sigma(s,w)\widehat{f_k\mu}\right)\right\vert^2(\xi)d\xi
\eeqs
with $\widetilde E(t,s)=\langle D\rangle^{m-l}\langle \cdot\rangle^{z+m-l}\langle D\rangle^\zeta E_{l-m}(t,s)$ family of FIOs of order $(z,\zeta).$
Now, using the well-known fact that the Fourier transform of a product is the ($(2\pi)^{-d}$ multiple of the)
convolution of the Fourier transforms, the 
property $f_k(-x)=f_k(x)$ (by the definition of $L^2_{\mu,s}$), that 
$\{f_k\}$ is an orthonormal system in $L^2_{\mu}$, and Bessel's inequality,
we get
\beqsn
(2\pi)^{-d}\sum_{k\in\N}\left\vert\mathcal F\left(\widetilde E(t,s)\sigma(s,w)\widehat{f_k\mu}\right)\right\vert^2(\xi)&=&(2\pi)^{-2d}\sum_{k\in\N}|\mathcal F\left(\widetilde E(t,s)\sigma(s,w)\right)\ast\widehat{\widehat{f_k\mu}}|^2(\xi)
\\
&=&(2\pi)^{-d}\sum_{k\in\N}|\mathcal F\left(\widetilde E(t,s)\sigma(s,w)\right)\ast f_k\mu|^2(\xi)
\\
&=&(2\pi)^{-d}\sum_{k\in\N}\left\vert\int_{\R^d}\left[\mathcal F\left(\widetilde E(t,s)\sigma(s,w)\right)\right](\xi-\eta) f_k(\eta)\mu(d\eta)\right\vert^2
\\
&\le &(2\pi)^{-d}\int_{\R^d}\left\vert\mathcal F\left(\widetilde E(t,s)\sigma(s,w)\right)\right\vert^2(\xi-\eta)\mu(d\eta).
\\
\eeqsn
Inserting this in \eqref{las}, and using the continuity of $\widetilde E$ on Sobolev-Kato spaces we finally get: 
\beqs\nonumber
||\Phi(t,s)||_{L_2^0(\mathcal H_\Xi, H^{z+m-l,\zeta})}^2&\leq&(2\pi)^{-d}
\int_{\R^d}\int_{\R^d}\langle\xi\rangle^{2(l-m)}\left\vert\mathcal F\left(\widetilde E(t,s)\sigma(s,w)\right)\right\vert^2(\xi-\eta)\mu(d\eta)d\xi
\\\nonumber
&= &(2\pi)^{-d}\int_{\R^d}\int_{\R^d}\langle\eta+\theta\rangle^{2(l-m)}\left\vert\mathcal F\left(\widetilde E(t,s)\sigma(s,w)\right)\right\vert^2(\theta)\mu(d\eta)d\theta
\\\nonumber
&\leq &(2\pi)^{-d}\left(\sup_{\theta\in\R^d}\int_{\R^d}\langle \theta+\eta\rangle^{2(l-m)}\mu(d\eta)\right)\int_{\R^d}\left\vert\mathcal F\left(\widetilde E(t,s)\sigma(s,w)\right)\right\vert^2(\theta)d\theta
\\\label{battezzata}
&=&(2\pi)^{-d}\left(\sup_{\theta\in\R^d}\int_{\R^d}\langle \theta+\eta\rangle^{2(l-m)}\mu(d\eta)\right)\|\mathcal F(\widetilde E(t,s)\sigma(s,w))\|_{L^2}^2
\\\nonumber
&\leq &\left(\sup_{\theta\in\R^d}\int_{\R^d}\langle \theta+\eta\rangle^{2(l-m)}\mu(d\eta)\right)C_{t,s}^2\|\sigma(s,w)\|_{z,\zeta}^2
\\\nonumber
&\leq &\left(\sup_{\theta\in\R^d}\int_{\R^d}\langle \theta+\eta\rangle^{2(l-m)}\mu(d\eta)\right)C_{t,s}^2C_s^2\left(1+\|w\|_{z+m-l,\zeta}\right)^2,
\eeqs
where $C_{t,s}$ stands for the norm in $\scrL(H^{z,\zeta},H^{z,\zeta})$ of the FIO $\widetilde E(t,s)\langle D\rangle^{-\zeta}\x^{-z}$, which,
by Theorem \ref{thm:compi}, has amplitude of order $(0,0)$. Since $\sigma\in \mathrm Lip(z,\zeta,m-l,0)$, $C_s$ is the constant in Definition \ref{def:lip}.
\end{proof}

%
\subsection{Function-valued solutions for semilinear hyperbolic equations of arbitrary order.}\label{subs:mainslin}

We are finally ready to deal with existence and uniqueness of a function-valued solution  for the Cauchy problem \eqref{cp} under conditions \eqref{roots2} and either \eqref{def:strict} or \eqref{def:constmult} or \eqref{def:involutive}.

In Theorem \ref{thm:semilinearweak} we study in parallel the strictly hyperbolic case and the weakly hyperbolic case with roots of constant multiplicity. In Theorem \ref{thm:semilinearinvolutive} we state a similar result for the involutive case.

%
\begin{theorem}\label{thm:semilinearweak} 
Let us consider the Cauchy problem \eqref{cp} for a hyperbolic SPDE \eqref{eq:SPDE}, where the partial differential operator $L$ of the form \eqref{elle2} satisfies \eqref{roots2}.
Moreover, assume that $L$ is weakly $SG$-hyperbolic with constant multiplcities, 
that is, $\caL_m$ satisfies \eqref{roots} and the characteristic roots $\tau_j$, $j=1,\dots, m$, 
can be divided into $n$ groups, $1\leq n< m$, of distinct and separated roots, 
in the sense that, possibly after a reordering of the $\tau_j$, $j=1,\dots, m$, there exist 
$l_1,\ldots l_n\in\N$ with $l_1+\ldots+l_n=m$ and $n$ sets
\[ G_1=\{\tau_1=\cdots=\tau_{\nu_1}\},\quad G_2=\{\tau_{\nu_1+1}=\cdots=\tau_{\nu_1+\nu_2}\},\quad \ldots \quad 
G_n=\{\tau_{m-\nu_n+1}=\cdots=\tau_{m}\},\]
satisfying \eqref{def:constmult} for some constant $C>0$. Assume also that $L$ is of Levy type, that is, with the notation of Corollary \ref{cor:resort}, it satisfies \eqref{eq:LC}.
Suppose that 
$\gamma,\sigma\in\mathrm{Lip_{loc}}(z,\zeta,m-l,0)$,
$z,\zeta\in\R$, in some sufficiently small open subset $U\subset H^{z+m-1,\zeta+m-1}(\R^d)\hookrightarrow H^{z+m-l,\zeta}(\R^d)$.
Finally, assume for the spectral measure that
\begin{equation}\label{eq:meascm}
\sup_{\eta\in\Rd}\int_\Rd \frac{1}{(1+|\xi+\eta|^2)^{m-l}}\mu(d\xi) < \infty,
\end{equation}
where $l=\max_{j=1,\dots,n}\nu_j$ is the maximum multiplicity of the roots of $L_m$.

Then, there exists a time horizon $0< T_0\leq T$ such that, for any choice of $u_j\in H^{z+m-1-j, \zeta+m-1-j}(\Rd)$, $0\leq j\leq m-1$,
$u_0\in U$,
the Cauchy problem \eqref{cp} admits a unique solution $u\in L^2([0,T_0]\times\Omega, H^{z+m-l,\zeta}(\R^d))$ satisfying 
\begin{equation}\label{eq:solslin}
	\begin{aligned}
	u(t,x)&=v_0(t,x)+\int_0^t\int_{\R^d}\Lambda(t,s,x,y)\gamma(s,y,u(s,y))\,dyds +
	\int_0^t\int_{\R^d}\Lambda(t,s,x,y)\sigma(s,y, u(s,y))\dot{\Xi}(s,y)\,dyds
	\end{aligned}
\end{equation}
where $\Lambda(t,s)$ is obtained in \eqref{eq:uexplicitbis}, the first integral in \eqref{eq:solslin} is a Bochner integral, and the second integral in \eqref{eq:solslin} is understood as the stochastic integral of the $H^{z+m-l,\zeta}(\R^d)$-valued stochastic process $E_{l-m}(t,\cdot)\sigma(\cdot,u(\cdot))$ with respect to the stochastic noise $\Xi$, in the sense explained in Subsection \ref{subH}.
\end{theorem}

\begin{remark}
	Notice that the noise $\Xi$
	defines a cylindrical Wiener process on 
	$(\mathcal H_\Xi(\R^d),\langle\cdot,\cdot\rangle_{\mathcal H_\Xi(\R^d)})$ 
	with values in $H^{z+m-l,\zeta}(\R^d)$, by Lemma \ref{lem:weightedpesz}. 
\end{remark}
\begin{remark} Notice that, if the correlation measure $\Gamma$ is absolutely continuous, then condition \eqref{eq:meascm} for $l=1$ (i.e., in the strictly hyperbolic case)
is equivalent to 
\begin{equation}\label{eq:measstbis}
  \int_\Rd \frac{1}{(1+|\xi|^2)^{m-1}}\mu(d\xi) < \infty,
\end{equation}
see \cite{peszat}. Condition \eqref{eq:measstbis} with $m=2$ on the spectral measure is the one needed for the existence and uniqueness of both a function-valued solution and a random-field solution to a second order SPDE well-known in literature, namely, the stochastic wave equation. 

Moreover, the same condition \eqref{eq:meascm}, with $l=1$. has been found in \cite{alessiandre}, looking for random-field solutions to linear strictly hyperbolic equations with uniformly bounded coefficients. The more general condition
\eqref{eq:meascm} is exactly the one obtained in \cite{ACSlinear}, looking for random-field solutions to
linear hyperbolic SPDEs with possibly unbounded variable coefficients. 
Thus, the class of the stochastic noises we can deal with if we want to obtain either a function-valued or a random-field solution of the Cauchy problem for an SPDE is described by \eqref{eq:meascm} for all $SG$-hyperbolic operators $L$. Condition \eqref{eq:meascm} can be understood as a {\emph {compatibility condition}} between the noise and the equation: as the order of the equation increases, we can allow for rougher stochastic noises $\Xi$; as the maximum multiplicity of the roots decreases (i.e., as the regularity of the operator $L$ increases), we can allow for rougher stochastic noises $\Xi$.
\end{remark}

\noindent We give here below a couple of examples of right-hand side that we can allow in \eqref{cp}. 

\begin{example}
	Let $\sigma(t, u)=u^2$. Then, $\sigma$ satisfies all the conditions
	required in Theorem \ref{thm:semilinearweak}. More generally,
	we can allow also $\sigma(t,u)=u^n$, $n\in\N$, $n>2$, see Remark 
	\ref{rem:intpowers}.
\end{example}

\begin{example}\label{ex:sigmanonlin}
	A class of explicitly $(t,x)$-dependent nonlinear
	stochastic coefficients which satisfy the
	requirements of Theorem \ref{thm:semilinearweak} are those of
	the form
	\beqs\label{expowerbis}
		\sigma(t,x,u)=\langle x\rangle^{l-m}\cdot\widetilde{\sigma}(t,u),
	\eeqs
	where $\widetilde{\sigma}\in \mathrm{Lip_{loc}}(z+m-l,\zeta,0,0)$. Indeed, the function $\sigma$ in \eqref{expowerbis} fulfills the assumptions of Theorem \ref{thm:semilinearweak}, being an element of $\mathrm{Lip_{loc}}(z,\zeta,m-\ell,0)$. In fact, for every $w$ in a sufficiently small subset $U\subset H^{z+m-l,\zeta}(\Rd)$, we have
	\[||\sigma(t,\cdot,w)||_{z,\zeta}=||\tilde\sigma(t,\cdot,w)||_{z+m-l,\zeta}\leq C(t)\left(1+||w||_{z+m-l,\zeta}\right),\]
	and the verification of $||\sigma(t,\cdot,w_1)-\sigma(t,\cdot, w_2)||_{z,\zeta}\leq C(t)||w_1-w_2||_{z+m-l,\zeta}$ follows similarly.
	\end{example}

The proof of Theorem \ref{thm:semilinearweak} consists of the following 4 steps:
\begin{enumerate}
\item factorization of the operator $L$;
\item reduction of \eqref{cp} to an equivalent first order system of the form \eqref{sys},
with the matrices $\kappa_1$ and $\kappa_0$ satisfying the assumptions described in Section \ref{sec:fio} above;
\item construction of the fundamental solution to \eqref{sys}, and then (formally) of the solution $u$ to \eqref{cp};
\item application of a fixed point scheme to obtain that the function-valued solution $u$ of \eqref{cp} is well-defined. 
\end{enumerate}
For steps (1), (2) and (3) we can rely on Proposition \ref{prop:mlpf}, Corollary \ref{cor:resort} and 
Lemma \ref{lem:Dt}, and on the procedure explained 
in \cite{Coriasco:998.3}.
We recall below the main aspects of this microlocal approach, for the convenience of the reader. A more detailed explanation, where the reduction to a first order system is performed first in the strictly hyperbolic case $m=n=1$ and then in the weakly hyperbolic one can be found in \cite{ACSlinear}.

\begin{proof}[Proof of Theorem \ref{thm:semilinearweak}]
We follow the computations of Subsection \ref{subs:reduction}. First, we perform a change of variable defining the ($nm$)-dimensional vector of unknowns $W$ having entries given by \eqref{eq:Wpbis}; our equation $Lu(t)=g(t,u)$, where formally $g(t,u):= \gamma(t,u) + \sigma(t,u)\dot{\Xi}(t)$
is so equivalent to the semilinear hyperbolic system of the first order \eqref{eq:Wmbis} (with $g(t,u)$ instead of $g(t)$) in the unknown $W$. 
The system in the unknown $W$ (with dimension $nm$) has the form 
\begin{equation}\label{syssl}
\begin{cases}
	(D_t - \Op(\kappa_1(t)) - \Op(\kappa_0(t)))W(t) = F(t,W(t))+G(t,W(t))\dot\Xi(t), & t\in [0,T], \\
	W(0)  = W_0,
\end{cases}
\end{equation}
with $\kappa_1\in C^\infty([0,T],S^{1,1})$ real-valued and diagonal, $\kappa_0\in C^\infty([0,T],S^{0,0})$, and $(nm)$-dimensional vectors $F(t,W(t))$, $G(t,W(t))$ given by
	\[
		F(t,W(t))=(\underbrace{\tilde F(t,W),\dots,\tilde F(t,W(t))}_{n\text{ times}})^t, \quad \tilde F(t,W(t))=(\underbrace{0,\dots,0}_{m-1\text{ times}},\gamma(t, W_1^{(1)}))^t,
	\]
	\[
		G(t,W(t))=(\underbrace{\tilde G(t,W),\dots,\tilde G(t,W(t))}_{n\text{ times}})^t, \quad \tilde G(t,W(t))=(\underbrace{0,\dots,0}_{m-1\text{ times}},\sigma(t, W_1^{(1)}))^t.
	\]
We still have that $W_0=\Op(b)U_0$, with a ($mn\times m$)-dimensional block-matrix symbol $b$ with structure analogous to \eqref{eq:bm} and entries with the same orders, so, by the assumptions of Theorem \ref{thm:semilinearweak}, we get $W_0\in H^{z,\zeta}$.

By Theorem \ref{thm:fundsolabs} we can formally construct, via Duhamel's formula, the ``mild solution" to \eqref{syssl}:
\[W(t)=E(t,0)W_0+i\ds\int_0^t E(t,s)F(s,W(s))ds+i\ds\int_0^t E(t,s)G(s,W(s))d\Xi(s),\quad t\in[0,T_0],
\] 
for a suitable $T_0\in(0, T]$. Now, we go back to the equation \eqref{eq:SPDE} to get its (formal) solution $u$. By Lemma \ref{lem:solregbis}, we know that
$u(t)$ is the first entry of the vector $\Op(\varUpsilon_n(t))W(t)$, and that the first row of $\varUpsilon_n(t)$ is a symbol of order $(l-m,l-m)$. Thus, exactly as in \eqref{eq:uexplicitbis}, we come (formally) to \eqref{eq:solslin}, that is we get
\beqsn
	u(t,x)&=&v_0(t,x)+\int_0^t\int_{\R^d}\Lambda(t,s,x,y)\gamma(s,y,u(s,y))\,dyds +\int_0^t\int_{\R^d}\Lambda(t,s,x,y)\sigma(s,y, u(s,y))\dot{\Xi}(s,y)\,dyds
\\
&=&v_0(t,x)+\int_0^tE_{l-m}(t,s)\gamma(s,u(s))ds +\int_0^t E_{l-m}(t,s)\sigma(s,u(s))\dot{\Xi}(s)ds,
\eeqsn
where $\displaystyle v_0\in \bigcap_{j\ge0} C^j([0,T_0], H^{z+m-l-j,\zeta+m-l-j})$ depends on the Cauchy data, and $\Lambda\in C^\infty(\Delta_{T_0},\caS^\prime)$ is, for any $(t,s)\in\Delta_{T_0}$, the Schwartz kernel of the Fourier integral operator family $E_{l-m}$, with amplitudes of order $(l-m,l-m)$.
We then construct the map $u\to \mathcal{T}u$ on $L^2([0,T_0]\times\Omega, H^{z+m-l,\zeta}(\R^d))$, defined as follows:
\beqs\label{calm}
\mathcal{T}u(t)&:=&
v_0(t)+\ds\int_0^t E_{l-m}(t,s)\gamma(s,u(s))ds+\ds\int_0^t E
_{l-m}(t,s)\sigma(s,u(s))dB_s, \quad t\in[0,T_0],
\\
\nonumber
&:=&v_0(t)+\mathcal{T}_1u(t)+\mathcal{T}_2u(t),
\eeqs
where the last integral on the right-hand side is understood as the stochastic integral of the stochastic process $E_{l-m}(t,\cdot)\sigma(\cdot,u(\cdot))\in L^2([0,T_0]\times\Omega, H^{z+m-l,\zeta})$ with respect to the cylindrical Wiener process $\{W_t(h)\}_{t\in[0,T], h\in H^{z+m-l,\zeta}}$ associated with the random noise $\Xi(t)$, which is well-defined by Lemma \ref{lem:weightedpesz} and takes values in $H^{z+m-l,\zeta}$. 

To prove that the solution \eqref{eq:solslin} of the Cauchy problem \eqref{cp} is indeed well-defined, we have to check that $$\mathcal{T}\colon L^2([0,T_0]\times\Omega, H^{z+m-l,\zeta}(\R^d))\longrightarrow L^2([0,T_0]\times\Omega, H^{z+m-l,\zeta})$$ is well-defined, it is Lipschitz continuous on $L^2([0,T_0]\times\Omega, H^{z+m-l,\zeta})$, and it becomes a contraction if we take $T_0$ small enough. Then, an application of Banach's fixed point Theorem will provide existence of a unique solution $u\in L^2([0,T_0]\times\Omega, H^{z+m-l,\zeta})$ satisfying $u=\mathcal{T}u$, that is \eqref{eq:solslin}.
\\
To verify that $\mathcal{T}u$ in \eqref{calm} belongs to $L^2([0,T_0]\times\Omega, H^{z+m-l,\zeta})$ for every $u\in L^2([0,T_0]\times\Omega, H^{z+m-l,\zeta})$ we notice that:
\\
-  $\displaystyle v_0\in \bigcap_{j\ge0} C^j([0,T_0], H^{z+m-l-j,\zeta+m-l-j})\subset L^2([0,T_0]\times\Omega, H^{z+m-l,\zeta});$
\\
- $\mathcal{T}_1u$ is in $L^2([0,T_0]\times\Omega, H^{z+m-l,\zeta})$; indeed, $\mathcal{T}_1u(t)$ is defined as the Bochner integral on $[0,t]$ of the function $s\to E_{l-m}(t,s)\gamma(s,u(s))$ with values in $L^2(\Omega, H^{z+m-l,\zeta})$, and, by the properties of Bochner integrals, the continuity of $E_{l-m}(t,s)$ on Sobolev-Kato spaces, and the fact that 
$\gamma\in \Lip(z,\zeta,m-l,0)$, we have
\beqsn%
\|\mathcal{T}_1u\|_{L^2([0,T_0]\times\Omega, H^{z+m-l,\zeta})}^2=&&\E\left[\ds\int_0^{T_0}\|\mathcal{T}_1u(t)\|_{z+m-l,\zeta}^2dt
\right]
=\ds\int_0^{T_0}\E\left[\left\|\ds\int_0^t E_{l-m}(t,s)(\gamma(s,u(s))ds\right\|_{z+m-l,\zeta}^2\right]dt
\\
&&\leq\ds\int_0^{T_0}\ds\int_0^t \E\left[\left\|E_{l-m}(t,s)(\gamma(s,u(s))\right\|_{z+m-l,\zeta}^2\right]dsdt
\\
&&\leq \ds\int_0^{T_0}\ds\int_0^t C_{t,s}^2\E\left[\left\|\gamma(s,u(s)))\right\|_{z,\zeta+l-m}^2\right]dsdt
\\
&&\leq \ds\int_0^{T_0}\ds\int_0^t C_{t,s}^2C_s^2\E\left[\left(1+\|u(s)\right\|_{z+m-l,\zeta+l-m})^2\right]dsdt
\\
&&\leq 2\left(\max_{0\leq s\leq t\leq T_0}C_{t,s}^2C_s^2\right)T_0\ds\int_0^{T_0} \left(1+\E\left[\left\|u(s)\right\|_{z+m-l,\zeta}^2\right]\right)ds
\\
&&=2C_{T_0}T_0(T_0+\|u\|^2_{L^2([0,T_0]\times\Omega, H^{z+l-m,\zeta})})<\infty;
\eeqsn
$ $
\\
- $\mathcal{T}_2u$ is in $L^2([0,T_0]\times\Omega, H^{z+m-l,\zeta})$, in view of the fundamental isometry \eqref{isomhilb}, 
Lemma \ref{lem:weightedpesz} and the fact that the expectation can be moved inside and outside time integrals, by Fubini's Theorem:
\beqsn
\|\mathcal{T}_2u\|_{L^2([0,T_0]\times\Omega, H^{z+m-l,\zeta})}^2&=&\E\left[\ds\int_0^{T_0}\|\mathcal{T}_2u(t)\|_{z+m-l,\zeta}^2dt\right]
=\ds\int_0^{T_0}\E\left[\left\|\ds\int_0^t E_{l-m}(t,s)\sigma(s,u(s))dW_s\right\|_{z+m-l,\zeta}^2\right]dt
\\
&=&\ds\int_0^{T_0}\ds\int_0^t \E\left[\left\|E_{l-m}(t,s)\sigma(s,u(s))\right\|_{L_0^2(\mathcal H_\Xi, H^{z+m-l,\zeta})}^2\right]dsdt
\\
&\leq&\ds\int_0^{T_0}\ds\int_0^t \E\left[ C^2_{(t,s)} \left(1+\|u(s)\|_{H^{z+m-l,\zeta}}\right)^2 \sup_{\eta\in\Rd}\int_\Rd \frac{1}{(1+|\xi+\eta|^2)^{m-l}}\mu(d\xi)\right]dsdt
\\
&=&\left(\sup_{\eta\in\Rd}\int_\Rd \frac{1}{(1+|\xi+\eta|^2)^{m-l}}\mu(d\xi)\right)\ds\int_0^{T_0}\ds\int_0^t C^2_{(t,s)}\E\left[ \left(1+\|u(s)\|_{H^{z+m-l,\zeta}}\right)^2\right]dsdt
\\
&\leq&2\left(\sup_{\eta\in\Rd}\int_\Rd \frac{1}{(1+|\xi+\eta|^2)^{m-l}}\mu(d\xi)\right) \left(\max_{0\leq s\leq t\leq T_0}C_{(t,s)}^2\right)T_0  \left(T_0+\ds\int_0^{T_0} \E\left[\left\|u(s)\right\|_{z+m-l,\zeta}^2\right]ds\right)
\\
&=& 2C_{T_0,m,l}T_0(T_0+\|u\|^2_{L^2([0,T_0]\times\Omega, H^{z+l-m,\zeta})})<\infty.
\eeqsn
Now, we take $u_1,u_2\in L^2([0,T_0]\times\Omega, H^{z+m-l,\zeta})$ and compute
\beqs\nonumber
\|\mathcal{T}u_1-\mathcal{T}u_2\|_{L^2([0,T_0]\times\Omega, H^{z+m-l,\zeta})}^2&\leq& 2\left(\|\mathcal{T}_1u_1-\mathcal{T}_1u_2\|_{L^2([0,T_0]\times\Omega, H^{z+m-l,\zeta})}^2+\|\mathcal{T}_2u_1-
\mathcal{T}_2u_2\|_{L^2([0,T_0]\times\Omega, H^{z+m-l,\zeta})}^2\right)
\\\label{pri}
&=&2\ds\int_0^{T_0}\E\left[\left\|\ds\int_0^t E_{l-m}(t,s)(\gamma(s,u_1(s))-\gamma(s,u_2(s)))ds\right\|_{z+m-l,\zeta}^2\right]dt
\\\label{se}
&+&2\ds\int_0^{T_0}\E\left[\left\|\ds\int_0^t E_{l-m}(t,s)(\sigma(s,u_1(s))-\sigma(s,u_2(s)))dB_s\right\|_{z+m-l,\zeta}^2\right]dt.
\eeqs
In the term \eqref{pri} here above we can move the expectation and the $(z+m-l,\zeta)-$norm inside the integral with respect to $s$. Then, by continuity of $E_{l-m}$ on Sobolev-Kato spaces, Definition \ref{def:lip}, and the embedding $H^{z+m-l,\zeta}\hookrightarrow H^{z+m-l,\zeta+l-m}$, we obtain 
\beqsn%
2\ds\int_0^{T_0}\E&&\left[\left\|\ds\int_0^t E_{l-m}(t,s)(\gamma(s,u_1(s))-\gamma(s,u_2(s)))ds\right\|_{z+m-l,\zeta}^2\right]dt
\\
&&\leq2\ds\int_0^{T_0}\ds\int_0^t \E\left[\left\|E_{l-m}(t,s)(\gamma(s,u_1(s))-\gamma(s,u_2(s)))\right\|_{z+m-l,\zeta}^2\right]dsdt
\\
&&\leq 2\ds\int_0^{T_0}\ds\int_0^t C_{t,s}^2\E\left[\left\|\gamma(s,u_1(s))-\gamma(s,u_2(s))\right\|_{z,\zeta+l-m}^2\right]dsdt
\\
&&\leq 2\ds\int_0^{T_0}\ds\int_0^t C_{t,s}^2C_s^2\E\left[\left\|u_1(s)-u_2(s)\right\|_{z+m-l,\zeta+l-m}^2\right]dsdt
\\
&&\leq 2\left(\max_{0\leq s\leq t\leq T_0}C_{t,s}^2C_s^2\right)T_0\ds\int_0^{T_0} \E\left[\left\|u_1(s)-u_2(s)\right\|_{z+m-l,\zeta}^2\right]ds
\\
&&=2C_{T_0}T_0\|u_1-u_2\|^2_{L^2([0,T_0]\times\Omega, H^{z+l-m,\zeta})}.
\eeqsn
To the term \eqref{se} we apply, here below, the fundamental isometry \eqref{isomhilb} to pass from the first to the second line, formula \eqref{battezzata} of Lemma \ref{lem:weightedpesz} to pass from the second to the third line, Definition \ref{def:lip} to pass from the third to the fourth line, and finally get:
\beqsn
2\ds\int_0^{T_0}\E&&\left[\left\|\ds\int_0^t E_{l-m}(t,s)(\sigma(s,u_1(s))-\sigma(s,u_2(s)))dB_s\right\|_{z+m-l,\zeta}^2\right]dt
\\
&&=2\ds\int_0^{T_0}\ds\int_0^t \E\left[\left\|E_{l-m}(t,s)(\sigma(s,u_1(s))-\sigma(s,u_2(s)))\right\|_{L_2^0(\mathcal H_\Xi, H^{z+m-l,\zeta})}^2\right]dsdt
\\
\\
&&\leq 2\ds\int_0^{T_0}\ds\int_0^t \E\left[
C^2_{t,s}\|\sigma(s,u_1(s))-\sigma(s,u_2(s))\|_{H^{z,\zeta}}^2\sup_{\eta\in\Rd}\int_\Rd \frac{1}{(1+|\xi+\eta|^2)^{m-l}}\mu(d\xi)\right]dsdt
\\
&&\leq 2\left(\sup_{\eta\in\Rd}\int_\Rd \frac{1}{(1+|\xi+\eta|^2)^{m-l}}\mu(d\xi)\right)\ds\int_0^{T_0}\ds\int_0^t C_{t,s}^2C_s^2\E\left[\left\|u_1(s)-u_2(s)\right\|_{z+m-l,\zeta}^2\right]dsdt
\\
&&\leq2C_{T_0}T_0\left(\sup_{\eta\in\Rd}\int_\Rd \frac{1}{(1+|\xi+\eta|^2)^{m-l}}\mu(d\xi)\right)\|u_1-u_2\|^2_{L^2([0,T_0]\times\Omega, H^{z+m-l,\zeta})}.
\eeqsn
Summing up, we have proved that
\[
\|\mathcal{T}u_1-\mathcal{T}u_2\|_{L^2([0,T_0]\times\Omega, H^{z,\zeta})}^2\leq  2C_{T_0}T_0\left(1+\sup_{\eta\in\Rd}\int_\Rd \frac{1}{(1+|\xi+\eta|^2)^{m-l}}\mu(d\xi)\right)\|u_1-u_2\|^2_{L^2([0,T_0]\times\Omega, H^{z+m-l,\zeta})} \,,
\]
that is, $\mathcal{T}$ is Lipschitz continuous on $L^2([0,T_0]\times\Omega, H^{z+m-l,\zeta})$. Moreover, in view of the assumption \eqref{eq:meascm},  if we take $T_0>0$ such that 
\beqs\label{time} 2C_{T_0}T_0\left(1+\sup_{\eta\in\Rd}\int_\Rd \frac{1}{(1+|\xi+\eta|^2)^{m-l}}\mu(d\xi)\right)< 1,\eeqs
then $\mathcal{T}$ becomes a strict contraction on $L^2([0,T_0]\times\Omega, H^{z+m-l,\zeta})$, and so it admits a unique fixed point $u=\mathcal{T}u$. That is, there exists a unique, well-defined solution of \eqref{cp}. To prove the estimate \eqref{time}, it is sufficient to take $T_0$ small enough, since the constant $C_{T_0}$ is continuously dependent on $T_0$.
The proof is complete.
\end{proof}

%

\subsection{The weakly hyperbolic case with involutive roots}\label{subs:inv}
We conclude the section with the statement of a result of existence and uniqueness of a solution to the Cauchy problem \eqref{cp} for the SPDE \eqref{eq:SPDE} in the more general case of involutive roots, cfr. \eqref{def:involutive}. 
With these even weaker hyperbolicity assumption we can still switch from \eqref{cp} to an equivalent first order system
\eqref{sys}, but at the price, as usual, of some further requirement on the lower order terms of the operator $L$. Namely, we ask, that $L$ admits a factorization \eqref{eq:Lfactor} with symbols $h_{jk}$, $j=1,\dots,m$, $k=1,\dots,l_j$, such that $h_{jk}\in C^\infty([0,T],S^{0,0})$.
Notice that this is automatically true in the case of strict hyperbolicity, and that only the request on the order of the symbols $h_{jk}$ has to be fulfilled in the case of hyperbolicity with constant multiplicities. 
We say, in the present case, that $L$ satisfies the \textit{strong Levi condition}, or, equivalently, that it is \textit{of strong Levi type}.
We state and discuss here below our further result, under the hypothesis \ref{hyp:in}.

\begin{theorem}\label{thm:semilinearinvolutive} 
Let us consider the Cauchy problem \eqref{cp} for an SPDE \eqref{eq:SPDE}, where the partial differential operator $L$ of the form \eqref{elle2} satisfies the hyperbolicity hypothesis \eqref{roots2}. Assume that $L$ is $SG$-hyperbolic with involutive roots, that is, 
all the roots of the principal part $L_m$ of $L$ are real-valued and
form an involutive system, in the sense of \eqref{def:involutive}. Moreover, assume that $L$ is of strong Levi type.
Suppose that 
$\gamma,\sigma\in\mathrm{Lip_{loc}}(z,\zeta,0,0)$,
$z,\zeta\in\R$, in some sufficiently small open subset $U\subset H^{z+m-1,\zeta+m-1}(\R^d)\hookrightarrow H^{z,\zeta}(\R^d)$.
Finally, assume that the spectral measure satisfies the compatibility condition
\begin{equation}\label{eq:measinv}
  \int_\Rd \mu(d\xi) < \infty.
\end{equation}
\\
Then, there exists a time horizon $0\leq T_0\leq T$ such that for any choice of $u_j\in H^{z+m-1-j, \zeta+m-1-j}(\Rd)$, $0\leq j\leq m-1$, $u_0\in U$,
the Cauchy problem \eqref{cp} admits a unique solution $u\in L^2([0,T_0]\times\Omega, H^{z,\zeta}(\R^d))$ satisfying
\beqsn
	u(t,x)&=&v_0(t,x)+\int_0^t\int_{\R^d}\Lambda(t,s,x,y)\gamma(s,y,u(s,y))\,dyds +\int_0^t\int_{\R^d}\Lambda(t,s,x,y)\sigma(s,y, u(s,y))\dot{\Xi}(s,y)\,dyds,
\eeqsn
where $\Lambda(t,s)$ is obtained through the Schwartz kernels of Fourier integral operators, the first integral is a Bochner integral, and the second integral is intended to be the stochastic integral of the $H^{z,\zeta}(\R^d)$-valued stochastic process $E_0(t,\cdot)\sigma(\cdot,u(\cdot))$ with respect to the stochastic noise 
$\Xi$.
\end{theorem}

\begin{remark}
	$\Xi$ defines a cylindrical Wiener process on $(\mathcal H_\Xi(\R^d),\langle\cdot,\cdot\rangle_{\mathcal H_\Xi(\R^d)})$ with values 
	in $H^{z,\zeta}$, by Lemma \ref{lem:weightedpesz}. 
\end{remark}

By the procedure explained in \cite{Coriasco:998.2}, see also \cite{Morimoto:1}, 
the Cauchy problem \eqref{cplin} is turned into an equivalent first order system
\eqref{sys} with real-valued, diagonal principal part. 
However, due to the failure of the ellipticity of the differences $\tau_j(t,x,\xi)-\tau_k(t,x,\xi)$, even in the sense of the
constant multiplicities, here we have no possibility to decouple (into blocks) the equations through a perfect diagonalizer
$\Op(\omega)$ as in Theorem \ref{thm:perfd}, and directly proceed, as in the case of 
scalar equations of order $1$, by means of Fourier integral operators. 

By Theorem \ref{thm:fundsol} from \cite{AleSandro},
we know that the fundamental solution of \eqref{sys} can be obtained as limit of a sequence of matrices of Fourier operators.
However, if we can write $E$ as a finite-sum expression of FIOs, also in the case of variable multiplicities and involutive characteristics, we can reduce to a system, construct formally its solution using its fundamental solution operator, go back to the scalar equation \eqref{eq:SPDE}, apply Banach fixed point theorem as done in the proof of Theorem \ref{thm:semilinearweak} and get existence of a unique solution $u\in L^2([0,T_0]\times\Omega, H^{z,\zeta}(\R^d))$. 
In \cite{AAC17} we are extending to the $SG$ case a result by Taniguchi \cite{taniguchi}, which allows to obtain $E$, in the case of involutive roots, through a finite sum of FIOs, modulo smoothing terms.

Another difference is that, in this case, we have no improvement in the decay and smoothness order
loss, as it is instead provided by the matrix-valued operators 
$\Op(\varUpsilon_j)$, $j=m$ or $j=n<m$, in the cases of strict hyperbolicity or of constant multiplicities, respectively. So, the symbols 
$c^1_p$ appearing
in the expected kernels of the \text{approximate solution operator} in the sense of \eqref{eq:kersolappbis}
will be of order $(0,0)$. This explains the more restrictive condition \eqref{eq:measinv}, which allows again to go through an
argument similar to those in the proof of Theorem \ref{thm:semilinearweak}.

The full proof of Theorem \ref{thm:semilinearinvolutive} requires a careful analysis of a number of technical details, 
to incorporate in the $SG$ theory the analog of the result by Taniguchi mentioned above, which will be achieved in \cite{AAC17}.
In order to keep the present exposition within a reasonable size, and not to heavily divert from the main objects of interests treated here, 
the proof of Theorem \ref{thm:semilinearinvolutive} will then appear elsewhere.

\subsection{Function-valued solutions and random-field solutions in the linear case.}\label{sec:comparison}
Consider now the special case of \eqref{cp}, with a $SG$-hyperbolic operator $L$ with constant multiplicities, where 
$\sigma(t,x,u(t,x))=\sigma(t,x)$ and $\gamma(t,x,u(t,x))=\gamma(t,x)$, $\gamma,\sigma\in C([0,T],H^{z,\zeta})$,
$z\ge0$, $\zeta>\frac{d}{2}$, $s\mapsto\caF(\sigma)(s)=\nu_s\in L^2([0,T], \caM_b)$, $\caM_b$ the space of
complex-valued measures with finite total variation. That is, we look at the Cauchy problem
\beqs\begin{cases}\label{cplinbis}
 Lu(t,x) = \gamma(t,x) + \sigma(t,x)\dot{\Xi}(t,x),\quad (t,x)\in(0,T]\times\R^d
 \\
 D_t^j u(0,x)=u_j(x),\quad x\in\R^d,\ 0\leq j\leq m-1,
\end{cases}
\eeqs 
for the linear SPDEs studied in \cite{ACSlinear}. 
Such (more restrictive) hypotheses imply 
$\gamma,\sigma\in \Lip(z,\zeta,r,\rho)\subset \Liploc(z,\zeta,r,\rho)$ for any $r,\rho\ge0$. In fact, recalling
Definition \ref{def:lip}, trivially:
\begin{itemize}
\item for every $w\in H^{z+r,\zeta+\rho}$, $t\in[0,T]$,
	$\|g(t,\cdot,w)\|_{z,\zeta}=\|g(t,\cdot)\|_{z,\zeta}\leq C(t)(1+\|w\|_{z+r,\zeta+\rho})$,
	with $C(t)=\|g(t,\cdot)\|_{z,\zeta}$;
\item for every $w,v\in H^{z+r,\zeta+\rho}$, $t\in[0,T]$,  
$\|g(t,\cdot,w)-g(t,\cdot,v)\|_{z,\zeta}\equiv 0 \leq C(t)\|w-v\|_{z+r,\zeta+\rho}$.
\end{itemize}
Applying Theorem \ref{thm:semilinearweak}, we obtain the existence and uniqueness of a function-valued solution for the linear
Cauchy problem \eqref{cplinbis}, which we here denote by $u_\fv$. 
Since in Theorem 4.12 of \cite{ACSlinear} we proved the existence and uniqueness of a random-field solution of
\eqref{cplinbis}, which we here denote by $u_\rf$, we now wish to compare it with $u_\fv$.

\begin{remark}\label{rem:comp}
	Notice that, in analogy with \eqref{eq:solslin}, $u_\rf$ satisfies
	\begin{equation}\label{eq:sollinbis}
		u_\rf(t,x)=v_0(t,x)+\int_0^t\int_{\R^d}\Lambda(t,s,x,y)\gamma(s,y)\,dyds +\int_0^t\int_{\R^d}\Lambda(t,s,x,y)\sigma(s,y)\dot{\Xi}(s,y)\,dyds.
	\end{equation}
	While the first two terms in the right-hand side of \eqref{eq:sollinbis} clearly coincide with the first two terms
	in the right-hand side of \eqref{eq:solslin}, the corresponding third, stochastic terms in \eqref{eq:solslin} and \eqref{eq:sollinbis}
	are defined in different ways. 
\end{remark}

We now prove that a random-field solution of \eqref{cplinbis} is also a function-valued solution.

\begin{proposition}\label{prop:cmp}
	Let $u_\rf$ and $u_\fv$ be the random-field solution
	and the function-valued solution of \eqref{cplinbis}, respectively, with $L$ $SG$-hyperbolic with constant multiplicities,
	$\gamma,\sigma\in C([0,T],H^{z,\zeta})$, $z\ge0$, $\zeta>\frac{d}{2}$, $s\mapsto\caF(\sigma)(s)=\nu_s\in L^2([0,T], \caM_b)$,
	$\caM_b$ the space of complex-valued measures with finite total variation.
	Then, $u_\rf=u_\fv=u$.
\end{proposition}
\begin{proof}
	Our analysis in \cite{ACSlinear} shows that $\Lambda \sigma\in\mathcal{P}_0$, the completion of the class $\caE$ of simple processes
	via the pre-inner product (defined for suitable $f,g$) 
	\[
		\langle f,g\rangle_{0}= \E\bigg[\int_0^T\int_\Rd \big(f(s)\ast\tilde{g}(s)\big)(x)\,\Gamma(dx) ds \bigg] = 
		\E\bigg[\int_0^T\int_\Rd [\caF f(s)](\xi)\cdot\overline{[\caF g(s)](\xi)}\,\mu(d\xi) ds\bigg].
	\]
	By Proposition 3.12 in \cite{dalangquer},
	it follows that the stochastic integrals 
	of $\Lambda \sigma$ with respect to the martingale measure associated with $\dot{\Xi}$ (considered in Section 4 of \cite{ACSlinear}), and with respect to the
	cylindrical Wiener process 
	considered in Section \ref{sec:nonlin} are equal.
	This proves that $u_\rf=u_\fv=u$, as claimed.
\end{proof}

%
%
%

\end{document}